\documentclass[a4paper, reqno]{amsart}
\usepackage{amsmath, amssymb, eucal, amscd, amstext, enumerate, mathrsfs, yfonts, amsfonts, verbatim, tikz, comment}

\usepackage[plainpages=false,colorlinks,hyperindex,pdfpagemode=None,bookmarksopen,linkcolor=red,citecolor=blue,urlcolor=blue]{hyperref}
\usepackage{pdflscape}

\newtheorem{thm}{Theorem}[section]
\newtheorem{lem}[thm]{Lemma}
\newtheorem{eg}[thm]{Example}
\newtheorem{prop}[thm]{Proposition}
\newtheorem{prop-defn}[thm]{Proposition and Definition}
\newtheorem{cor}[thm]{Corollary}
\newtheorem{defn}[thm]{Definition}
\newtheorem{rem}[thm]{Remark}

\newtheorem{thm2}{Theorem}

\numberwithin{equation}{section}

\newcommand{\ti}{\tilde}
\newcommand{\smnoind}{\smallskip\noindent}
\newcommand{\CL}{\mathcal{L}}
\newcommand{\la}{\langle}
\newcommand{\ra}{\rangle}
\newcommand{\CU}{\mathcal{U}}
\newcommand{\CA}{\mathcal{A}}

\newcommand{\CC}{\mathcal{C}}

\newcommand{\CI}{\mathcal{I}}
\newcommand{\BR}{\mathbb{R}}

\newcommand{\BN}{\mathbb{N}}

\newcommand{\id}{\mathrm{id}}
\newcommand{\BC}{\mathbb{C}}

\newcommand{\BK}{\mathbb{K}}

\newcommand{\SG}{\mathscr{G}}

\newcommand{\GL}{G\mathcal{L}}
\newcommand{\Ad}{\mathrm{Ad}}

\newcommand{\idp}{Q}
\newcommand{\prj}{P}
\newcommand{\ivt}{W}
\newcommand{\fin}{\mathrm{fin}}
\newcommand{\Inv}{\mathcal{V}}
\newcommand{\OGL}{\mathrm{GL}}

\newcommand{\OO}{\mathrm{O}}
\newcommand{\smallT}{{\scriptstyle\top}}

\begin{document}

\title {Analytic bundle structure on the idempotent manifold\footnote{This work is supported by the National Natural Science Foundation of China (11471168) and (11871285).}}

\author{Chi-Wai Leung and Chi-Keung Ng}

\address[Chi-Wai Leung]{Department of Mathematics, The Chinese University of Hong Kong, Hong Kong.}
\email{cwleung@math.cuhk.edu.hk}

\address[Chi-Keung Ng]{Chern Institute of Mathematics and LPMC, Nankai University, Tianjin 300071, China.}
\email{ckng@nankai.edu.cn}

\keywords{Infinite dimensional Grassmannian, Idempotents, Banach bundles, Affine-Banach spaces, Tangent bundles}

\subjclass[2010]{Primary: 46T05, 57N20, 58D15}

\date{\today}

\begin{abstract}
Let $X$ be a (real or complex) Banach space, and $\mathcal{I}(X)$ be the set of all (non-zero and non-identity) idempotents; i.e., bounded linear operators on $X$ whose squares equal themselves. 
We show that the Banach submanifold $\mathcal{I}(X)$ of $\mathcal{L}(X)$ is a \emph{locally trivial analytic affine-Banach bundle} over the Grassmann manifold $\mathscr{G}(X)$, via the map $\kappa$ that sends $Q\in \mathcal{I}(X)$ to $Q(X)$, such that the affine-Banach space structure on each fiber is the one induced from $\mathcal{L}(X)$ (in particular, every fiber is an affine-Banach subspace of $\mathcal{L}(X)$). 

Using this, we show that if $K$ is a real Hilbert space, then the assignment 
$$(E,T)\mapsto T^*\circ P_{E^\bot} + P_{E}, \quad \text{ where } E\in \mathscr{G}(K)\text{ and } T\in \mathcal{L}(E,E^\bot),$$ induces a bi-analytic bijection from the total space of the tangent bundle, $\mathbf{T}(\mathscr{G}(K))$, of $\mathscr{G}(K)$ onto $\mathcal{I}(K)$ (here, $E^\bot$ is the orthogonal complement of $E$, $P_E\in \mathcal{L}(K)$ is the orthogonal projection onto $E$, and $T^*$ is the adjoint of $T$). 
Notice that this bi-analytic bijection is an affine map on each tangent plane. 
\end{abstract}
\maketitle

\section{Introduction}

\medskip

The Grassmannian of a finite dimensional vector space is a very well-studied object. 
This manifold is important in both pure and applied mathematics (see e.g. \cite{BJM, BKT, CM, DeSt-H, FP, FKT, GS, GLS, Ito, KW, LXZ, NRC, ZBLST} for some recent accounts on it). 
There are two main streams of generalizations of the Grassmannian to the infinite dimensional case. 
The first one was introduced by Douady in \cite{Doua} (see also \cite{Kaup75} and \cite{Upm}). 
In this case, the Grassmannian of a (either real or complex) Banach space $X$ is the set of complemented subspaces of $X$ equipped with a canonical (respectively, real or complex) \emph{analytic Banach manifold} structure such that a local chart around a subspace is given by the Banach space of continuous linear operators from that subspace to a complement of it. 
Another approach was first appeared in the work of Porta and Recht in \cite{PR}. 
In this approach, the Grassmannian of a Banach algebra $D$ is defined to be the set of equivalence classes of idempotents under certain equivalence relation, and is equipped with the quotient topology. 
In the particular case when $D$ is a $C^*$-algebra, the Grassmannian can be identified, as a topological space, with the set of (self-adjoint) projections, and the later is a \emph{real analytic} Banach submanifold of $D$ (see e.g., \cite{AMaj}, \cite{BG} and \cite{CPR}; see also \cite{Chu}, \cite{CI} and \cite{Kaup} for the generalizations to $JB$-algebras and $JB^*$-triples). 
A connection between the two approaches was obtain \cite{AM}, where it was shown that there is a \emph{real bi-analytic} bijection from the Grassmannian of a \emph{complex} Hilbert space $H$ (in the sense of Douady) to the set of self-adjoint projections of the $C^*$-algebra $\CL(H)$ of continuous linear operators on $H$. 

\medskip

We will follows \cite{Doua} for the definition of the Grassmann manifold $\SG(X)$ of a (either real or complex) Banach space $X$.
We denote by $\CI(X)$ the set of all non-zero proper idempotents in the Banach manifold $\CL(X)$ of all bounded linear operators. 
Notice that ``idempotents'' were also called ``projections'' in some literature, but we prefer the term ``idempotents'' in order to distinguish them with ``self-adjoint projections'' in the case of Hilbert spaces (which will also be considered in this article). 
We define $\kappa: \CI(X) \to \SG(X)$ as follows: 
\begin{equation}\label{eqt:def-kappa}
\kappa(\idp):=\idp(X)\qquad (\idp\in \CI(X)).
\end{equation}

\medskip

The main results of this paper can be summarized in the following (see Theorem \ref{thm:main}, Proposition \ref{prop:loc-tri-cont-Ban} and Proposition \ref{prop:not-complex-anal}). 

\medskip

\begin{thm2}
Let $X$ be a real or complex Banach space, which is not one-dimensional. 

\smnoind
(a) Under the Banach submanifold structure on $\CI(X)$ induced from $\CL(X)$, one knows that 
$\kappa: \CI(X) \to \SG(X)$ is a locally trivial \emph{(respectively, real or complex) analytic} affine-Banach bundle, such that the affine-Banach space structures on the fibers of $\kappa$ are the ones induced from $\CL(X)$.

\smnoind
(b) There exist equivalent Banach space structures on the fiber of $\kappa$, under which $\kappa: \CI(X) \to \SG(X)$  becomes a locally trivial \emph{continuous} Banach bundle. 

\smnoind
(c) In the case when $X$ is a complex Banach space, there \emph{can never} exist Banach space structures on the fiber of $\kappa$ such that $\kappa: \CI(X) \to \SG(X)$ becomes a locally trivial complex \emph{analytic} Banach bundle. 
\end{thm2}

\medskip

In the case of a real Hilbert space $K$, we also obtain the following: 
\begin{quotation}
there is an analytic immersion from the total space of the tangent bundle of $\SG(K)$ to $\CL(K)$ such that the restriction of this immersion on each fiber is affine. 
\end{quotation}
More precisely, elements in the tangent bundle of $\SG(K)$ can be identified with a pair $(E,T)$, where $E\in \SG(K)$ and $T$ is a bounded linear operator from $E$ to 
the orthogonal complement $E^\bot$  of $E$. 
The following is obtained in Theorem \ref{thm:Hil-tang} and Corollary \ref{cor:tan-embed-trivial-bundle}(a). 

\medskip

\begin{thm2}
Let $K$ be a real Hilbert space, and  $\mathbf{T}(\SG(K))$ be the tangent bundle of $\SG(K)$. 

\smnoind
(a) The assignment $\idp \mapsto \big(\idp(X),\bar \prj_{\idp(X)^\bot}\circ \idp^*|_{\idp(X)}\big)$, where $\bar \prj_{\idp(X)^\bot}:K\to \idp(X)^\bot$ is the orthogonal projection,  is a bi-analytic bijection from $\CI(K)$ onto $\mathbf{T}(\SG(K))$ such that for each $E\in \SG(K)$, this bijection is an affine map from $\{\idp\in \CI(K): \idp(K) = E \}$ onto the tangent plane over $E$. 

\smnoind
(b) The assignment $(E,T)\mapsto (E, T\circ\bar \prj_{E})$ is a fiberwise linear bi-analytic map from $\mathbf{T}(\SG(K))$ onto the Banach subbundle
$\big\{(E,S): E\in \SG(K);  S\in \CL^{E^\bot}(K,E^\bot)\big\}$
of the trivial Banach bundle $\big(\SG(K)\times \CL(K), \SG(K), \kappa_0\big)$; here $\CL^F(K,F):= \big\{S\in \CL(K): T(F) = \{0\}; T(K)\subseteq F \big\}$.   
\end{thm2}

\medskip

On our way, we also obtain that $\CI(X)$ (under the norm topology) is canonically homeomorphic to the following subspace of the product topological space $\SG(X)\times \SG(X)$ (see Corollary \ref{cor:I=GxG}): 
$$\{(E,F)\in \SG(X)\times \SG(X): E \text{ and }F \text{ are complements to each other} \}.$$
Furthermore, similar to the corresponding fact for $\SG(X)$, we will show, in Corollary \ref{cor:I(X)-union-homog-sp}, that each orbit in $\CI(X)$ under the canonical action by the Banach Lie group $\GL(X)$ of continuous invertible operators on $X$ is a clopen subset, and can be identified bi-analytically with a homogeneous space of $\GL(X)$. 
We will also verify in Proposition \ref{prop:fin-dim-conn} that, for any $n\in \BN$, the set $\{\idp\in \CI(X): \dim \idp(X) = n \}$ is a connected component of $\CI(X)$.

\medskip

Using the above, for $n\in \BN$ and $k\in \{1,\dots,n\}$, if the map 
\begin{equation*}\label{eqt:G-O}
\upsilon:\OGL_n(\BR)/ \OGL_k(\BR)\times \OGL_{n-k}(\BR) \to \OO_n/ \OO_k\times \OO_{n-k}
\end{equation*}
is given by the Gram-Schmidt process (on column vectors), then 
$\upsilon$ induces a locally trivial real analytic vector bundle structure on the homogeneous space $\OGL_n(\BR)/ \OGL_k(\BR)\times \OGL_{n-k}(\BR)$, which can be identified with the tangent bundle of  $\OO_n/ \OO_k\times \OO_{n-k}$ (see Example \ref{eg:Gram-Sch}(b)). 

\bigskip

\section{Notations and Preliminary}

\medskip

Let us begin this paper by giving some notation. 
Throughout this article, $\BK$ is either the real field $\BR$ or the complex field $\BC$. 
If $X$ and $Y$ are $\BK$-Banach spaces, we denote by $\CL(X,Y)$ the Banach space of all continuous $\BK$-linear operators from $X$ to $Y$. 
We will also denote $\CL(X):= \CL(X,X)$.  
Moreover, the identity map in $\CL(X)$ will be denoted by $I_X$, and sometimes by $I$ if no confusion arise. 

\medskip

Unless specified otherwise, by $\BK$-Banach manifolds, we mean $\BK$-analytic Banach manifolds, in the sense of \cite{Upm} and \cite{Chu12}. 



\medskip

Throughout this article, 
$\SG(X)$ is the collection of all \emph{non-zero proper} complemented subspaces of $X$. 
For any $E\in \SG(X)$, we denote $F\smallT E$ if $F$ is a complement of $E$, and put 
\begin{equation}\label{eqt:def-C-E}
\CC_E:=  \{F\in \SG(X): F\smallT E\}.
\end{equation}
We set, as in \cite[p.44-46]{Upm}, 
\begin{equation}\label{eqt:defn-p-E-F}
p_{E,F} = q_{E,F}^{-1},
\end{equation}
where $q_{E,F}:\CL(E,F)\to \CC_F$ is the bijection  given by 
\begin{equation}\label{eqt:defn-q-E-F}
q_{E,F}(T):= (I+T)(E) \qquad (T\in \CL(E,F)). 
\end{equation}
There is a Hausdorff metrizable topology on $\SG(X)$ such that $\CC_F$ is an open subset of $\SG(X)$ and $p_{E,F}: \CC_F\to \CL(E,F)$ is a homeomorphism. 
Moreover, 
\begin{equation}\label{eqt:atlas-Upm}
\big\{\big(\CC_{F_0}, p_{E_0,F_0}, \CL(E_0,F_0)\big): E_0,F_0\in \SG(X); F_0\smallT E_0 \big\}
\end{equation}
constitutes an analytic atlas for a $\BK$-Banach manifold structure on $\SG(X)$. 
When equipped with this structure, $\SG(X)$ is called the \emph{Grassmann manifold} of $X$. 

\medskip

A subset $A\subseteq X$ is called a \emph{$\BK$-affine-Banach subspace} if $A - a_0$ is a $\BK$-Banach subspace of $X$ for one (equivalently, for every) element $a_0\in A$. 
Moreover, a map $S$ from $A$ to a $\BK$-Banach space $Y$ is said to be \emph{$\BK$-affine} if the assignment
$$S^{a_0}: a - a_0\mapsto S(a) - S(a_0)$$ 
is a $\BK$-linear map on $A-a_0$.
We say that $A$ is \emph{isometrically affine isomorphic} to $Y$ if there is a $\BK$-affine bijection $T:A\to Y$ which preserves the metrics, i.e., $\|T(a) - T(b)\|_Y = \|a-b\|_X$ ($a,b\in A$). 

\medskip

We denote by $\CA(A,B)$ the set of all continuous $\BK$-affine maps from $A$ to a $\BK$-affine-Banach subspace $B$ of $Y$. 
It is clear that $\CA(A,B)$ is a $\BK$-vector space. 
The function $\|\cdot\|_{a_0}$ defined by
$$\|T\|_{a_0} := \|T^{a_0}\| + \|T(a_0)\| \qquad (T\in \CA(A,B))$$ 
is a complete norm on $\CA(A,B)$. 

\medskip

The following well-known fact ensures a default Banach space structure on $\CA(A,B)$ up to Banach space isomorphism.

\medskip

\begin{lem}\label{lem:affine}
For any $a_0, a_1\in A$, the two norms $\|\cdot\|_{a_0}$ and $\|\cdot\|_{a_1}$ are equivalent. 
\end{lem}

\medskip

This paper mainly concerns with affine-Banach bundles. 
We need to consider such a general notion (instead of the more well-known notion of Banach bundles) because the set of idempotents naturally forms an affine-Banach bundle. 
Since affine-Banach bundles are not well-documented, let us give its precise definition below. 

\medskip

\begin{defn}\label{defn:aff-Ban}
Let $\Omega$ and $\Upsilon$ be Hausdorff spaces. 
Let $\kappa: \Upsilon \to \Omega$ be a continuous surjection.

\smnoind
(a) Then $(\Upsilon, \Omega, \kappa)$ is called a \emph{locally trivial continuous $\BK$-affine-Banach bundle} (respectively, \emph{locally trivial continuous $\BK$-Banach bundle}) if the following conditions are satisfied:
\begin{enumerate}[\ \ \ B1)]
	\item for each $\omega\in \Omega$, the subset 
	$$\Upsilon_{\omega}:= \kappa^{-1}(\omega)$$ 
	is homeomorphic to a $\BK$-affine-Banach subspace of a $\BK$-Banach space (respectively, homeomorphic to a  $\BK$-Banach space); 
	
	\item for each $\omega_0\in \Omega$, there exist an open neighborhood $V_{\omega_0}\subseteq \Omega$ of $\omega_0$ as well as a bi-continuous bijection $\Theta_{\omega_0}: V_{\omega_0} \times \Upsilon_{\omega_0} \to \kappa^{-1}(V_{\omega_0})$ such that $\Theta_{\omega_0}|_{\{\omega\}\times \Upsilon_{\omega_0}}$ is a $\BK$-affine map (respectively, $\BK$-linear map) onto $\Upsilon_{\omega}$, for every $\omega\in V_{\omega_0}$;
	
	\item for $\omega_1,\omega_2\in \Omega$ with $V_{\omega_1}\cap V_{\omega_2}\neq \emptyset$, the map $\varphi: V_{\omega_1}\cap V_{\omega_2} \to \CA\big(\Upsilon_{\omega_1},  \Upsilon_{\omega_2}\big)$ defined by 
	$$\varphi(\omega)(x) := \Pi_2\big(\Theta_{\omega_2}^{-1}\circ \Theta_{\omega_1}(\omega, x)\big)\quad (\omega\in V_{\omega_1}\cap V_{\omega_2}; x\in \Upsilon_{\omega_1})$$ 
	is continuous, where $\Pi_2$ is the projection onto the second coordinate. 
\end{enumerate}

\smnoind
(b) Suppose that $\Omega$ and $\Upsilon$ are $\BK$-Banach manifolds. 
Then $(\Upsilon, \Omega, \kappa)$ is called a \emph{locally trivial $\BK$-analytic affine-Banach bundle} (respectively, \emph{locally trivial  $\BK$-analytic Banach bundle}), if $\kappa$ is $\BK$-analytic and the same requirements as in part (a) hold with the terms ``bi-continuous'' and ``continuous'' in (B2) and (B3) being replaced by ``$\BK$-bi-analytic'' and ``$\BK$-analytic'', respectively. 
\end{defn}

\medskip

In the case of a Banach bundle, Condition (B3) is equivalent to the corresponding statement when  $\CA\big(\Upsilon_{\omega_1},  \Upsilon_{\omega_2}\big)$ is replaced by $\CL\big(\Upsilon_{\omega_1},  \Upsilon_{\omega_2}\big)$ (because of Lemma \ref{lem:affine}).

\medskip

\medskip

\emph{We may occasionally use the term ``locally trivial continuous affine-Banach bundle'' and ``locally trivial  analytic affine-Banach bundle'' etc, if the underlying field $\BK$ is understood. }

\medskip

A map $\rho$ from an open subset $V_0\subseteq \Omega$ to $\Upsilon$ is called a \emph{local cross section} if 
$$\kappa\big(\rho(\omega)\big) = \omega \qquad (\omega\in V_0).$$ 
In the case when $V_0 =\Omega$, we say that $\rho$ is a \emph{global cross section}.

\medskip

For a continuous (respectively, analytic) Banach bundle, the constant zero map is obviously a continuous (respectively, analytic) global cross section. 
The following proposition tells us that the only obstruction for an affine-Banach bundle to be a  Banach bundle is the existence of global continuous (respectively, analytic) cross sections. 

\medskip

\begin{prop}\label{prop:cont-sect-F-Ban-bundle}
(a) Let $(\Upsilon, \Omega, \kappa)$ be a locally trivial continuous affine-Banach bundle. 
If there is a continuous global cross section $\rho: \Omega\to \Upsilon$, then there exist Banach space structures on all the fibers of $\kappa$, which are isometrically affine isomorphic to the original affine-Banach space structures on the fibers, such that $(\Upsilon, \Omega, \kappa)$ becomes a locally trivial continuous Banach bundle. 

\smnoind
(b) Let $(\Upsilon, \Omega, \kappa)$ be a locally trivial analytic affine-Banach bundle. 
If there is an analytic global cross section $\rho: \Omega\to \Upsilon$, then $(\Upsilon, \Omega, \kappa)$ is as a locally trivial analytic Banach bundle, under a Banach space structure on each fiber of $\kappa$ that is isometrically affine isomorphic to the original affine-Banach space structure on the fiber. 
\end{prop}

\medskip

Although this proposition could be a known fact (at least in the finite dimensional case), we nevertheless give a brief account of it here. 
Fix $\omega_0\in \Omega$. 
The affine-Banach space $\Upsilon_{\omega_0}$ becomes a Banach space when equipped with the following structure:
\begin{equation*}\label{eqt:vs-str}
\|x\|_{\omega_0} := \|x - \rho(\omega_0)\|, \quad \alpha\odot_{\omega_0} x := \alpha x + (1-\alpha)\rho(\omega_0) \quad \text{and} \quad x\oplus_{\omega_0} y:= x + y -\rho(\omega_0), 
\end{equation*}
where $x,y\in \Upsilon_{\omega_0}$, $\alpha\in \BK$ and $\|\cdot\|$ is the norm on the Banach space containing $\Upsilon_{\omega_0}$.
Let $V_{\omega_0}$, $\Theta_{\omega_0}$ and $\Pi_2$ be as in Definition \ref{defn:aff-Ban}. 
Set $\zeta_{\omega_0}: V_{\omega_0} \to \Upsilon_{\omega_0}$ to be the map
$\Pi_2\circ \Theta_{\omega_0}^{-1}\circ \rho|_{V_{\omega_0}}$.  
By the continuity (respectively, analyticity) assumption on  $\Theta_{\omega_0}^{-1}$ and  $\rho$, the map $\zeta_{\omega_0}$ is continuous (respectively, analytic). 
Define $\Psi_{\omega_0}: V_{\omega_0}\times \Upsilon_{\omega_0}\to \kappa^{-1}(V_{\omega_0})$ by 
$$\Psi_{\omega_0}(\omega, y) := \Theta_{\omega_0}\big(\omega, y \oplus_{\omega_0} \zeta_{\omega_0}(\omega)\big)  \qquad (\omega\in V_{\omega_0}; y\in \Upsilon_{\omega_0}).$$
Clearly,  $\Psi_{\omega_0}$ is continuous (respectively, analytic). 
It is not hard to verify that $\Psi_{\omega_0}$ is fiberwise linear.
Furthermore, one can check that
$$\Psi_{\omega_0}^{-1}(x) = \Theta_{\omega_0}^{-1}(x) \boxplus \big(\kappa(x), 2\rho(\omega_0) - \zeta_{\omega_0}(\kappa(x))\big) \qquad (x\in \kappa^{-1}(V_{\omega_0})),$$
where $(\omega,a)\boxplus (\omega, b):= (\omega, a\oplus_{\omega_0} b)$ ($\omega\in V_{\omega_0}; a,b\in \Upsilon_{\omega_0}$). 
This means that $\Psi_{\omega_0}$ admits a continuous (respectively, analytic) inverse. 
Hence, Condition (B2) is established. 
On the other hand, suppose that $\omega_1,\omega_2\in \Omega$ with $V_{\omega_1}\cap V_{\omega_2}\neq \emptyset$. 
We define two functions $\xi: V_{\omega_1} \to \CA(\Upsilon_{\omega_1}, \Upsilon_{\omega_1})$ and $\eta:  V_{\omega_2} \to \CA(\Upsilon_{\omega_2}, \Upsilon_{\omega_2})$ by 
$$\xi(\omega)(x) := x \oplus_{\omega_1} \zeta_{\omega_1}(\omega) \quad \text{ and } \quad 
\eta(\omega)(z) := z \oplus_{\omega_2} (-1)\odot_{\omega_2} \zeta_{\omega_2}(\omega).$$
Observe that for $a\in \Upsilon_{\omega_1}$, if $\chi_a\in \CA(\Upsilon_{\omega_1}, \Upsilon_{\omega_1})$ is the map  given by $\chi_a(x) := x \oplus_{\omega_1} a$, then $\chi:a\mapsto \chi_a$ is continuous and affine (and hence is analytic).
Since $\zeta_{\omega_1}$ and $\zeta_{\omega_2}$ are continuous (respectively, analytic), we know that both $\xi$ and $\eta$, being the compositions of $\zeta_{\omega_i}$ with $\chi$, are continuous  (respectively, analytic).
If $\varphi$ is as in Condition (B3), then for any $\omega\in V_{\omega_1}\cap V_{\omega_2}$ and $x\in \Upsilon_{\omega_1}$, one has  
\begin{align*}
\Pi_2\big(\Psi_{\omega_2}^{-1}\big(\Psi_{\omega_1}(\omega,x)\big)\big)
& = \varphi(\omega)\big(x \oplus_{\omega_1} \zeta_{\omega_1}(\omega)\big) \oplus_{\omega_2} (-1)\odot_{\omega_2} \zeta _{\omega_2}(\omega) 
 = \big(\eta(\omega)\circ \varphi(\omega)\circ \xi(\omega)\big)(x). 
\end{align*}
Therefore, Condition (B3) still holds when the maps $\Theta_{\omega_1}$ and $\Theta_{\omega_2}$ are replaced by $\Psi_{\omega_1}$ and $\Psi_{\omega_2}$, respectively. 

\medskip

\begin{rem}\label{rem:gen-aff-Ban}
%
(a) Suppose that $(\Upsilon, \Omega, \kappa)$ is a locally trivial continuous affine-Banach bundle such that $\Omega$ is paracompact.
Then one can see from the discussion before Proposition \ref{prop:loc-tri-cont-Ban} that there always exists a continuous global cross section. 
This means that when $\Omega$ is paracompact, there is no different between locally trivial continuous affine-Banach bundles over $\Omega$ and locally trivial continuous Banach bundles over $\Omega$. 

\smnoind
(b) Note that the ``$\BK$-analyticity'' in the above means the existence of  ``local power series expansions'' (as in Definition 1.6 and p.36 of \cite{Upm}). We say that a map from an open subset of a $\BK$-Banach space to another $\BK$-Banach space is \emph{$\BK$-differentiable} if it is Frechet differentiable. 
One can also defined the notion of \emph{locally trivial $\BK$-differentiable affine-Banach bundles} if one replaces the terms ``bi-continuous''  and ``continuous'' in Conditions (B2) and (B3) of Definition \ref{defn:aff-Ban}(a) by ``$\BK$-bi-differentiable'' and ``$\BK$-differentiable'', respectively. 
Clearly, a locally trivial $\BK$-analytic affine-Banach bundle is a  locally trivial $\BK$-differentiable affine-Banach bundle. 
The corresponding statement as Proposition \ref{prop:cont-sect-F-Ban-bundle} is also valid for $\BK$-differentiable affine-Banach bundles. 
\end{rem}

\medskip

Let us end this section with the following obvious fact for later reference.

\medskip

\begin{lem}\label{lem:square-zero}
If $R\in \CL(X)$ satisfying $R^2 = 0$, then $I + R$ is invertible with inverse $I - R$. 
\end{lem}

\bigskip

\section{$\CI(X)$ is a locally trivial analytic affine-Banach bundle}

\medskip

From now on, $X$ is a $\BK$-Banach space with $\dim_\BK X > 1$ (could be infinite). 
Our main concern is the following set
$$\CI(X):= \big\{\idp\in \CL(X)\setminus \{0,I\}: \idp^2 = \idp \big\},$$
with different structures induced from $\CL(X)$.
For any $\idp\in \CI(X)$, it is obvious that both $Q(X)$ and $\ker \idp$ belongs to $\SG(X)$, and that $\idp(X)\smallT \ker \idp$ (i.e., they are complement of each other). 
Conversely, if $E\in \SG(X)$ and $F\in \CC_E$ (see \eqref{eqt:def-C-E}), there is a unique element  $\idp^F_E\in \CI(X)$ with 
$$\idp^F_E(X) = E \quad \text{and}\quad \ker \idp^F_E = F.$$

\medskip

For every $F_0\in \SG(X)$, let us denote 
$$\CI(X)^{F_0}:= \{\idp\in \CI(X): \ker \idp = F_0\}, \quad \CI(X)_{F_0}:= \{\idp\in \CI(X): \idp(X) = F_0\}$$
as well as
$$\CL^{F_0}(X, F_0):= \{T\in \CL(X): T(X)\subseteq F_0  \text{ and } T(F_0)= (0) \}.$$
The starting point of this paper is the following easy observation. 

\medskip

\begin{lem}\label{lem:I-F-affine}
Let $X$ be a Banach space and $E_0, F_0\in \SG(X)$ with $E_0 \smallT F_0$. 
Then $\CI(X)^{F_0} = \CL^{F_0}(X,F_0) + \idp^{F_0}_{E_0}$, and $\CI(X)_{F_0} = \CL^{F_0}(X,F_0) + \idp^{E_0}_{F_0}$. 
\end{lem}
\begin{proof}
By considering the bijection $\idp \mapsto I -\idp$ from $\CI(X)^{F_0}$ onto $\CI(X)_{F_0}$, one only needs to verify the first equality. 
In fact, if $E\in \CC_{F_0}$ (see \eqref{eqt:def-C-E}), then 
\begin{equation*}
\big(\idp^{F_0}_E - \idp^{F_0}_{E_0} \big)(x) = \idp^{F_0}_E(x) - x = - \idp^E_{F_0}(x)\in F_0 \qquad (x\in E_0).
\end{equation*}
From this, as well as the fact that $\idp^{F_0}_{E} - \idp^{F_0}_{E_0}$ vanish on $F_0$, we obtain 
$\big(\idp^{F_0}_{E} - \idp^{F_0}_{E_0}\big)(X)\subseteq {F_0}.$
Hence, $\idp - \idp^{F_0}_{E_0}\in \CL^{F_0}(X,F_0)$, for every $\idp\in \CI(X)^{F_0}$. 

Conversely, consider an element $R\in \CL^{F_0}(X,F_0)$. 
It is not hard to check that $(I+R)(E_0)\in \CC_{F_0}$. 
For any $x\in E_0$, as $R(x)\in F_0$, one has
\begin{equation*}
\idp^{F_0}_{(I+R)(E_0)}(x) = \idp^{F_0}_{(I+R)(E_0)}(x+ R(x)) = x +  R(x),
\end{equation*}
which gives 
\begin{equation}\label{eqt:QFET-T}
R(x) = \idp^{F_0}_{(I+R)(E_0)}(x) - \idp^{F_0}_{E_0}(x).
\end{equation}
Furthermore, as $R$, $\idp^{F_0}_{(I+R)(E_0)}$ and $\idp^{F_0}_{E_0}$ all vanish on $F_0$, we know that $R = \idp^{F_0}_{(I+R)(E_0)} - \idp^{F_0}_{E_0}$. 
\end{proof}

\medbreak

Consider $F_0\in \CC_{E_0}$ and $E_0,E_1\in \CC_{F_0}$. 
We have
\begin{equation}\label{eqt:QFE(E0)}
\idp^{F_0}_{E_1}(E_0)  = \idp^{F_0}_{E_1}(E_0+F_0)  = E_1. 
\end{equation}
Moreover, Lemma \ref{lem:I-F-affine} allows us to define a map $\pi_{E_0,F_0}$ from $\CC_{F_0}$ onto $\CL^{F_0}(X,{F_0})$ via
\begin{equation}\label{eqt:defn-pi-E-F}
\pi_{E_0,F_0}(E):=\idp^{F_0}_E - \idp^{F_0}_{E_0} \qquad (E\in \CC_{F_0}). 
\end{equation}
Set $T:= p_{E_0,F_0}(E_1)\in \CL(E_0,F_0)$ (see \eqref{eqt:defn-p-E-F}). 
Then $E_1=(I+T)(E_0)$  (see \eqref{eqt:defn-q-E-F}). 
Thus, applying \eqref{eqt:QFET-T} to $R:=T\circ \bar\idp^{F_0}_{E_0}\in \CL^{F_0}(X,F_0)$ (where $\bar\idp^{F_0}_{E_0}$ is the idempotent $\idp^{F_0}_{E_0}$ regarding as a map from $X$ to $E_0$), we get 
$$T(x) = \idp^{F_0}_{E_1}(x) - \idp^{F_0}_{E_0}(x) \qquad (x\in E_0).$$
In other words, 
\begin{equation}\label{eqt:rel-q-pi}
p_{E_0,F_0} = \Lambda_{E_0,F_0}\circ \pi_{E_0,F_0},  
\end{equation}
where 
\begin{equation}\label{eqt:defn:L-F0-X-F0}
\Lambda_{E_0,F_0}: \CL^{F_0}(X,{F_0}) \to \CL(E_0,F_0)
\end{equation}
is the Banach space isomorphism given by restrictions. 
Notice that the inverse of $\Lambda_{E_0,F_0}$ is given by 
compositions of elements in $\CL(E_0,F_0)$  with the map $\bar \idp^{F_0}_{E_0}:X\to E_0$ as in the above. 
Furthermore, \eqref{eqt:defn-pi-E-F} implies 
\begin{equation}\label{eqt:converse-map}
\pi_{E_0,F_0}(E_1) = - \pi_{E_1,F_0}(E_0).
\end{equation}
On the other hand, by Relation \eqref{eqt:rel-q-pi}, the analytic atlas of $\SG(X)$ as in \eqref{eqt:atlas-Upm} can be rewritten as:
\begin{equation}\label{eqt:anal-loc-atlas}
\big\{\big(\CC_{F_0}, \pi_{E_0,F_0}, \CL^{F_0}(X,{F_0})\big): E_0, F_0\in \SG(X); F_0\smallT E_0 \big\}.
\end{equation}
One good point of this atlas is that elements in $\CL^{F_0}(X, F_0)$ are nilpotent operators of degree two, and Lemma \ref{lem:square-zero} applies to them. 
Another benefit of this atlas is the following result, which gives a clear picture of the topology on $\SG(X)$. 
In particular, we know that $E_k$  converges to $E_0$ is basically the same as $\idp^{F_0}_{E_k}(x)$  converges to $\idp^{F_0}_{E_0}(x)$ in a uniform way on all bounded subsets. 

\medskip

\begin{cor}\label{cor:conv-in G(X)}
Suppose that $\{E_k\}_{k\in \BN}$ is a sequence in $\SG(X)$ and $E_0\in \SG(X)$. 
Then $E_k\to E_0$ if and only if for every $F_0\in \CC_{E_0}$ (equivalently, there exists $F_0\in \CC_{E_0}$), there exists $k_0\in \BN$ such that $E_k\in \CC_{F_0}$ when $k\geq k_0$ and that $\big\|\idp^{F_0}_{E_k} - \idp^{F_0}_{E_0}\big\| \to 0$. 
\end{cor}

\medskip

Let $\kappa: \CI(X)\to \SG(X)$ be the surjection as given in \eqref{eqt:def-kappa}. 
Lemma \ref{lem:I-F-affine} tells us that the fiber $\kappa^{-1}(E)$ (which coincides with $\CI(X)_E$) is an affine-Banach subspace of $\CL(X)$, for every $E\in \SG(X)$.
The main result of this section is that under these affine-Banach space structures on the fibers and the Banach submanifold structure on $\CI(X)$ induced from $\CL(X)$, 
$(\CI(X), \SG(X), \kappa)$ is a locally trivial analytic affine-Banach bundle.

\medskip

For the proof of this statement, we need to consider the canonical actions of $\GL(X)$ on $\CI(X)$ and on $\SG(X)$. 
Indeed, for any $\ivt\in \GL(X)$ and $E,F\in \SG(X)$ with $F \smallT E$, one easily sees that 
\begin{equation}\label{eqt:defn-Ad-act}
\Ad(\ivt)(\idp^F_E):= \ivt\circ \idp^F_E\circ \ivt^{-1} = \idp^{\ivt(F)}_{\ivt(E)},
\end{equation}
and this produces an action of $\GL(X)$ on $\CI(X)$.
On the other hand, 
\begin{equation}\label{eqt:def-alpha}
\alpha (\ivt, E):= \ivt(E)
\end{equation}
induces an action of $\GL(X)$ on $\SG(X)$. 
It is well-known that the two actions $\alpha$ and $\Ad$ are $\BK$-analytic. 

\medskip

We also need the following easy fact for the proof of the main theorem.

\medskip

\begin{lem}\label{lem:comp-affine-subsp}
Let $Y$ be a $\BK$-Banach space, and let $A,B\subseteq Y$ be two $\BK$-affine-Banach subspaces. 
Suppose that $B-b_0\in \SG(Y)$ for an element $b_0\in B$. 
Then there is a continuous affine map $\Gamma: \CL(Y) \to \CA(A,B)$ such that whenever $T\in \CL(Y)$ satisfying $T(A)\subseteq B$, one has $\Gamma(T) = T|_A$. 
\end{lem}

\medskip

In fact, let us pick an element $D\in \CC_{B-b_0}$, and  define a map  $\ti \idp\in \CA(Y,B)$ by $\ti \idp(y):= \idp^D_{B-b_0}(y-b_0) + b_0$ ($y\in Y$). 
If we set $\Gamma(T):= \ti \idp\circ T|_A$ ($T\in \CL(Y)$), then clearly, $\Gamma$ is a continuous affine map satisfying the requirement. 

\medskip

A final piece of well-known information that we need is the following. 
For $E,F\in \SG(X)$ with $F\smallT E$, one has $\CL^E(X,E)\cap \CL^F(X,F) = \{0\}$ (since $E+F=X$).
We identify $\CL^E(X,E)\oplus \CL^F(X,F)$ with the sum of the two subspaces in $\CL(X)$. 
If we define a map $\Delta_{E,F}:\CL(X)\to \CL(X)$ by 
$\Delta_{E,F}(T):= \idp^{F}_{E}\circ T\circ \idp^{E}_{F}$ $(T\in \CL(X))$,  
then both $ \Delta_{E,F}$ and $\Delta_{E,F} + \Delta_{F,E}$ are idempotents with 
\begin{equation}\label{eqt:image-Delta}
\Delta_{E,F}\big(\CL(X)\big) = \CL^{E}(X, E) \quad \text{and}\quad (\Delta_{E,F} + \Delta_{F,E})\big(\CL(X)\big) = \CL^{E}(X, E)\oplus \CL^{F}(X, F).
\end{equation}
Consequently, both $\CL^{E}(X, E)$ and 
$\CL^{E}(X, E)\oplus \CL^{F}(X, F)$ are complemented subspaces of $\CL(X)$.
In the following we will denote elements in $\CL^E(X,E)\oplus \CL^F(X,F)$ in either the form $(R,S)$ or $R+S$.

\medskip

Before presenting the main theorem of this section, let us first give an outline of its proof, and give some remarks. 

\medskip

We will begin by showing that $\kappa$ is continuous. 
We will then show that $(\CI(X), \SG(X), \kappa)$ satisfies Conditions (B1) and (B2) of Definition \ref{defn:aff-Ban}(a) (i.e., the continuous case).
From this, we construct an analytic atlas for $\CI(X)$ (see \eqref{eqt:anal-atlas-IX}).
This produces a $\BK$-analytic Banach manifold structure on $\CI(X)$ compatible with the norm topology. 
We will then verify that the inclusion map from $\CI(X)$ to $\CL(X)$ is  an analytic immersion. 
Hence, $\CI(X)$ is a Banach submanifold of $\CL(X)$, under the above Banach manifold structure.

\medskip

Note that as elements in $\CI(X)$ satisfies the algebraic relation $\idp^2 - \idp = 0$, it is a known fact that $\CI(X)$ is a closed submanifold of $\CL(X)$ when $\dim X<\infty$. 
Even though $\CI(X)$ being a submanifold of $\CL(X)$ may also be a known fact in the infinite dimensional case, in order to verify that $\CI(X)$ is a locally trivial analytic affine-Banach bundle over $\SG(X)$, we need to use the explicitly atlas for $\CI(X)$ as in \eqref{eqt:anal-atlas-IX}. 
This atlas will also be needed in the later part of this article.  
One good feature of this analytic atlas for $\CI(X)$ is that it is ``algebraic'' in nature (see \eqref{eqt:mu-E-F} and \eqref{eqt:mu-E-F-inv}). 

\medskip

Finally, we will establish Conditions (B2) and (B3) of Definition \ref{defn:aff-Ban}(b) (i.e. the analytic case). 
Observe that if the fiber over each point in $\SG(X)$ were a Banach space, then one might use Proposition 1.2 in Chapter 3 of \cite{Lang} to simplify the argument.
However, since we are in the affine-Banach setting, we give a more direct argument here. 

\medskip

\begin{thm}\label{thm:main}
Let $X$ be a $\BK$-Banach space. 
If $\CI(X)$ is equipped with the $\BK$-analytic Banach submanifold structure induced from $\CL(X)$, then $\big(\CI(X), \SG(X), \kappa\big)$ is a locally trivial $\BK$-analytic affine-Banach bundle, such that the affine-Banach space structure on $\kappa^{-1}(E)$ is the one induced from $\CL(X)$, for every $E\in \SG(X)$. 
\end{thm}
\begin{proof}
We first establish the continuity of $\kappa$. 
For this, let us consider a sequence $\big\{\idp^{F_n}_{E_n} \big\}_{n\in \BN}$ in $\CI(X)$ converging to $\idp^{F_0}_{E_0} \in \CI(X)$; i.e., 
\begin{equation*}\label{eqt:I-Q-Q-invert}
\big\|I - \big(\idp^{E_n}_{F_n} + \idp^{F_0}_{E_0}\big)\big\| = \big\|\big(\idp^{F_n}_{E_n} + \idp^{E_0}_{F_0}\big) - I\big\| = 
\big\|\idp^{F_n}_{E_n} - \idp^{F_0}_{E_0}\big\| \to 0.
\end{equation*}
As $\CL(X)$ is a unital Banach algebra, we know that $\idp^{E_n}_{F_n} + \idp^{F_0}_{E_0}$ and $\idp^{F_n}_{E_n} + \idp^{E_0}_{F_0}$ are eventually invertible, and we assume that they are invertible for all $n\in \BN$. 
The relation $\big(\idp^{F_n}_{E_n} + \idp^{E_0}_{F_0}\big)(X) = X$ and $\ker \big(\idp^{E_n}_{F_n} + \idp^{F_0}_{E_0}\big) = \{0\}$ will then imply that $E_n\in \CC_{F_0}$. 

By Corollary \ref{cor:conv-in G(X)}, we need to show that $\idp^{F_0}_{E_n}\to \idp^{F_0}_{E_0}$.
Indeed, it is clear that 
\begin{equation}\label{eqt:QQ=Q}
\idp^{F_n}_{E_n}\circ \idp^{F_0}_{E_n} = \idp^{F_0}_{E_n}
\quad \text{and}\quad \idp^{F_0}_{E_0} \circ \idp^{F_0}_{E_n} = \idp^{F_0}_{E_0}.
\end{equation}
Moreover, one has  $\big\| \idp^{F_{n}}_{E_{n}} - \idp^{F_0}_{E_0}\big\| \leq 1/2$ when $n$ is large, and in this case,  
\begin{align*}
\big\| \idp^{F_0}_{E_n}(x) \big\| & = \big\|\idp^{F_n}_{E_n}\big(\idp^{F_0}_{E_n}(x)\big)\big\| 
 \leq  \big\|\idp^{F_0}_{E_0}\big(\idp^{F_0}_{E_n}(x)\big)\big\| + \big\| \idp^{F_0}_{E_n}(x) \big\|/2 \qquad (x\in X), 
\end{align*}
which implies $\big\| \idp^{F_0}_{E_n} \big\| \leq 2 \big\| \idp^{F_0}_{E_0} \big\|$. 
It follows that 
$$\lambda_0 := {\sup}_{n\in \BN}\ \! \big\|\idp^{E_n}_{F_0}\big\| 
< \infty.$$ 
Now, for any $\epsilon > 0$, there is $n_0$ such that $\big\|\idp^{F_n}_{E_n} - \idp^{F_0}_{E_0}\big\| < \epsilon$ whenever $n\geq n_0$. 
It then follows from $\idp^{F_0}_{E_0}\circ\idp^{E_n}_{F_0} = 0$ and Relation \eqref{eqt:QQ=Q} that for any  $x\in X$,
\begin{align*}
\big\|\idp^{F_n}_{E_n}(x) - \idp^{F_0}_{E_n}(x)\big\| 
& 
= \big\|\idp^{F_n}_{E_n}\big(\idp^{E_n}_{F_0}(x)\big) - \idp^{F_0}_{E_0}\big(\idp^{E_n}_{F_0}(x)\big)\big\| < \epsilon \lambda_0\|x\|.
\end{align*}
Thus, we have $\big\|\idp^{F_0}_{E_n}-\idp^{F_n}_{E_n}\big\|\to 0$. 	
This, together with $\|\idp^{F_n}_{E_n} - \idp^{F_0}_{E_0}\| \to 0$, implies the required convergence, and hence $\kappa$ is continuous.

On the other hand, it follows from Lemma \ref{lem:I-F-affine} that $\kappa^{-1}(E)$ is an affine-Banach subspace of $\CL(X)$, for each $E\in \SG(X)$; in other words, Condition (B1) in Definition \ref{defn:aff-Ban}(a) is satisfied. 

Let us now show that $(\CI(X), \SG(X), \kappa)$ satisfies Condition (B2) of Definition \ref{defn:aff-Ban}(a). 
We will do this via the construction of an analytic (and hence continuous) local right inverse for the evaluation maps from $\GL(X)$ to orbits of the action $\alpha$ as in \eqref{eqt:def-alpha}. 
For this, we fix arbitrary elements $E_0, F_0\in \SG(X)$ with $E_0\smallT F_0$. 
Consider $E\in \CC_{F_0}$. 
Since $\idp^{F_0}_E - \idp^{F_0}_{E_0}\in \CL^{F_0}(X, F_0)$, we know from Lemma \ref{lem:square-zero} that $\big(I + \big(\idp^{F_0}_{E} - \idp^{F_0}_{E_0}\big)\big)^{-1} =I - \big(\idp^{F_0}_{E} - \idp^{F_0}_{E_0}\big)$. 
Therefore, if we set 
\begin{equation}\label{eqt:def-Xi}
\Xi_{E_0,F_0}(E):= \idp^{F_0}_E + \idp^{E_0}_{F_0},
\end{equation}
then $\Xi_{E_0,F_0}(E)\in \GL(X)$ and 
\begin{equation}\label{eqt:Xi-inv}
\Xi_{E_0,F_0}(E)^{-1}  
= \idp^E_{F_0} + \idp^{F_0}_{E_0} 
= 2I - \Xi_{E_0,F_0}(E).
\end{equation} 
Recall that $\big(\CC_{F_0}, \pi_{E_0,F_0}, \CL^{F_0}(X, F_0)\big)$ is a local chart for $\SG(X)$ near $E_0$ (see \eqref{eqt:anal-loc-atlas}). 
As 
\begin{equation}\label{eqt:Xi-pi-inv}
\Xi_{E_0,F_0}\big(\pi_{E_0,F_0}^{-1}(R)\big) 
=  I + R\qquad (R\in \CL^{F_0}(X, F_0)),
\end{equation}
the map $\Xi_{E_0,F_0}: \CC_{F_0} \to \GL(X)$ is analytic. 
Moreover, one has 
\begin{equation}\label{eqt:Xi-E0}
\alpha\big(\Xi_{E_0,F_0}(E),E_0\big) 
=\idp^{F_0}_E(E_0) =\idp^{F_0}_E(X)= E.
\end{equation}
Consequently, $\Xi_{E_0,F_0}$ is an analytic local right inverse for the evaluation map at $E_0$ from $\GL(X)$ to the orbit $\alpha(\GL(X),E_0)$. 

We now define 
$\Theta_{E_0,F_0}: \CC_{F_0} \times \CI(X)_{E_0}\to \kappa^{-1}(\CC_{F_0})$ by  
\begin{equation}\label{eqt:defn-Theta}
\Theta_{E_0,F_0}(E, \idp) := \Ad\big(\Xi_{E_0,F_0}(E)\big)(\idp) \qquad (E\in \CC_{F_0}, \idp\in \CI(X)_{E_0}).
\end{equation}
For $E\in \CC_{F_0}$ and $F\in \CC_{E_0}$, we know from \eqref{eqt:defn-Ad-act} and \eqref{eqt:Xi-E0} that
\begin{equation}\label{eqt:Theta-E0-F0}
\Theta_{E_0,F_0}\big(E,\idp^F_{E_0}\big) = \idp^{\Xi_{E_0,F_0}(E)(F)}_{\Xi_{E_0,F_0}(E)(E_0)} = \idp^{(\idp^{F_0}_{E}+ \idp^{E_0}_{F_0})(F)}_E.
\end{equation}
This implies that $\Theta_{E_0,F_0}$ is well-defined 
and injective (since $\Xi_{E_0,F_0}(E)$ is invertible). 
Furthermore, it is clear from the definition that $\Theta_{E_0,F_0}$ is fiberwise affine. 
On the other hand, for any $E'\in \CC_{F_0}$ and $F'\in \CC_{E'}$, if we set $F'':= \Xi_{E_0,F_0}(E')^{-1}(F')$, then $F''\in \CC_{E_0}$ because $\Ad\big(\Xi_{E_0,F_0}(E')^{-1}\big)\big(\idp^{F'}_{E'}\big) = \idp^{F''}_{E_0}$, and  so, $\Theta_{E_0,F_0}(E', \idp^{F''}_{E_0}) = \idp^{F'}_{E'}$. 
This means that $\Theta_{E_0,F_0}$ is surjective. 

In the following, we establish the bi-continuity of $\Theta_{E_0,F_0}$. 
Let $\{F_n\}_{n\in \BN}$ and 
$\{E_n\}_{n\in \BN}$ be sequences in $\CC_{E_0}$ and in $\CC_{F_0}$, respectively. 
If $\{E_n\}_{n\in \BN}$ converges to $E\in \CC_{F_0}$ and $\{\idp^{F_n}_{E_0} \}_{n\in \BN}$ converges to $\idp^F_{E_0}\in \CI(X)_{E_0}$, then it follows 
from the continuity of $\Xi_{E_0,F_0}$ and \eqref{eqt:Xi-inv} that 
$$\Theta_{E_0,F_0}\big(E_n, \idp^{F_n}_{E_0}\big) = \Xi_{E_0,F_0}(E_n)\idp^{F_n}_{E_0}\big(2I - \Xi_{E_0,F_0}(E_n)\big) \to \Theta_{E_0,F_0}\big(E, \idp^{F}_{E_0}\big).$$ 
Conversely, assume that $\Theta_{E_0,F_0}\big(E_n, \idp^{F_n}_{E_0}\big)\to \Theta_{E_0,F_0}\big(E, \idp^{F}_{E_0}\big)$. 
The continuity of $\kappa$ and Relation \eqref{eqt:Theta-E0-F0} give $E_n \to E$ and $\Xi_{E_0,F_0}(E_n)(F_n) \to \Xi_{E_0,F_0}(E)(F)$. 
Consequently, the continuity of $\Xi_{E_0,F_0}$ tells us that 
$$F_n = \big(2I - \Xi_{E_0,F_0}(E_n)\big)\big(\Xi_{E_0,F_0}(E_n)(F_n)\big) \to F.$$
From this, we know that $\idp^{F_n}_{E_0} \to \idp^{F}_{E_0}$. 

We are now ready to construct a Banach manifold structure on $\CI(X)$ that is compatible with the norm topology. 
For every $E_0\in \SG(X)$ and $F_0\in \CC_{E_0}$, we consider the bijection $\phi_{E_0,F_0}: \CI(X)_{E_0}\to \CL^{E_0}(X,E_0)$ induced by Lemma \ref{lem:I-F-affine}; namely,
$$\phi_{E_0,F_0}(\idp) := \idp - \idp^{F_0}_{E_0}  \qquad (\idp\in \CI(X)_{E_0}).$$ 
Set $\mu_{E_0,F_0}: \kappa^{-1}(\CC_{F_0}) \to \CL^{F_0}(X,F_0)\oplus \CL^{E_0}(X, E_0)$ to be the map $\big(\pi_{E_0,F_0}\times \phi_{E_0,F_0}\big)\circ \Theta_{E_0,F_0}^{-1}$. 
We claim that 
\begin{equation}\label{eqt:anal-atlas-IX}
\big\{\big(\kappa^{-1}(\CC_{F_0}), \mu_{E_0,F_0}, \CL^{F_0}(X,F_0)\oplus \CL^{E_0}(X, E_0)\big): E_0,F_0\in \SG(X); F_0\smallT E_0 \big\}
\end{equation}
is an analytic atlas for $\CI(X)$. 

In fact, $\mu_{E_0,F_0}$ is a homeomorphism since $\pi_{E_0,F_0}$, $\phi_{E_0,F_0}$ and $\Theta_{E_0,F_0}$ are homeomorphisms.
Notice also that if $E\in \CC_{F_0}$ and $F\in \CC_E$, then we have 
\begin{align}
\mu_{E_0,F_0}\big(\idp^{F}_E\big) 
& = \big(\idp^{F_0}_E - \idp^{F_0}_{E_0}, \Ad(\idp^E_{F_0} + \idp^{F_0}_{E_0})(\idp^{F}_E) - \idp^{F_0}_{E_0}\big)\label{eqt:mu-E-F-1}\\
& = \big(\idp^{F_0}_E - \idp^{F_0}_{E_0},
\idp^{F_0}_{E_0}\idp^{F}_E\idp^{E_0}_{F_0}\big)\label{eqt:mu-E-F}, 
\end{align}
because of \eqref{eqt:QFE(E0)}, 
\eqref{eqt:QQ=Q} and \eqref{eqt:Theta-E0-F0}. 
Moreover, for $(R,S)\in  \CL^{F_0}(X,F_0)\oplus \CL^{E_0}(X, E_0)\big)$, one has 
\begin{align}
\mu_{E_0,F_0}^{-1}(R,S) 
& = (I+R)\big(S+\idp^{F_0}_{E_0}\big)(I-R) \label{eqt:mu-E-F-inv-1}\\
& = S - SR + \idp^{F_0}_{E_0} + RS - RSR + R, \label{eqt:mu-E-F-inv} 
\end{align}
because of \eqref{eqt:Xi-pi-inv}, Lemma \ref{lem:square-zero} as well as the facts that $R(X)\subseteq F_0$ and $R\idp^{F_0}_{E_0} = R$. 

Assume now that $E_1,F_1\in \SG(X)$ with $E_1\smallT F_1$ such that $\kappa^{-1}(\CC_{F_0})\cap \kappa^{-1}(\CC_{F_1})\neq \emptyset$. 
Consider an arbitrary element $(R,S)\in \mu_{E_0,F_0}\big(\kappa^{-1}(\CC_{F_0}\cap \CC_{F_1})\big)$.  
By Lemma \ref{lem:I-F-affine}, there exist unique elements $E_{R}\in \CC_{F_0}$ 
with $R = \idp^{F_0}_{E_{ R}} - \idp^{F_0}_{E_0}$. 
These produce, via 
\eqref{eqt:mu-E-F-1} as well as \eqref{eqt:mu-E-F-inv-1},
\begin{align*}\label{eqt:mu-mu-inv}
\mu_{E_1,F_1}\big(\mu_{E_0,F_0}^{-1} (R, S)\big) 
& = \Big(\idp^{F_1}_{E_{R}} - \idp^{F_1}_{E_1}, \Ad\big(\idp^{E_{R}}_{F_1} + \idp^{F_1}_{E_1}\big)\big(\Ad\big(I+R\big)\big(S + \idp^{F_0}_{E_0}\big)\big) - \idp^{F_1}_{E_1}\Big).
\end{align*}
Since $\big\{\big(\CC_{F}, \pi_{E,F}, \CL^{F}(X,{F})\big): E,F\in \SG(X); F\smallT E \big\}$ is an analytic atlas of $\SG(X)$, the map from $\pi_{E_0,F_0}(\CC_{F_0}\cap \CC_{F_1})$ onto $\pi_{E_1,F_1}(\CC_{F_0}\cap \CC_{F_1})$ that sends $\idp^{F_0}_{E} - \idp^{F_0}_{E_0}$ to $\idp^{F_1}_{E} - \idp^{F_1}_{E_1}$  is analytic. 
In other words, if we set 
\begin{equation*}\label{eqt:def-Phi}
\Phi(R):= \idp^{F_1}_{E_{R}}- \idp^{F_1}_{E_1},
\end{equation*}
then $(R, S) \mapsto \Phi(R)$ is an analytic map from $\pi_{E_0,F_0}(\CC_{F_0}\cap \CC_{F_1})\times \CL^{E_0}(X,E_0)$ to $\pi_{E_1,F_1}(\CC_{F_0}\cap \CC_{F_1})$. 
Thus, the assignment
\begin{equation*}\label{eqt:RS}
(R,S) \mapsto \big(I - \Phi(R)\big)\big(I+R\big)\big(S + \idp^{F_0}_{E_0}\big)\big(I-R\big)\big(I+ \Phi(R)\big) - \idp^{F_1}_{E_1}
\end{equation*}
 is also analytic. 
Consequently, $\mu_{E_1,F_1}\circ \mu_{E_0,F_0}^{-1}$ is analytic on $\mu_{E_0,F_0}(\CC_{F_0}\cap \CC_{F_1})$, and \eqref{eqt:anal-atlas-IX} is an analytic atlas for $\CI(X)$. 

Next, we will show that, when equipped with the above manifold structure, $\CI(X)$ is a Banach submanifold of $\CL(X)$, by verifying that the inclusion map $\iota: \CI(X) \to \CL(X)$ is an analytic immersion. 
Let us fix $E_0\in \SG(X)$ and $F_0\in \CC_{E_0}$. 
Set $\theta_{E_0,F_0}$ to be the map $\iota \circ \mu_{E_0,F_0}^{-1}: \CL^{F_0}(X, F_0)\oplus \CL^{E_0}(X, E_0) \to \CL(X)$. 
By \eqref{eqt:mu-E-F-inv}, one has
\begin{align}\label{eqt:theta-E0-F0}
\theta_{E_0,F_0}(R,S) & 
= S - SR + \idp^{F_0}_{E_0} + RS - RSR + R.
\end{align}
Hence, $\iota$ is analytic. 
Consider 
$$\mathbf{T}(\iota): \mathbf{T}(\CI(X)) \to \mathbf{T}(\CL(X))$$ 
to be the map between the respective tangent bundles induced by $\iota$. 
We need to show that the map 
$$\mathbf{T}_{\idp^{F_0}_{E_0}}(\iota) = 
\theta_{E_0,F_0}'(0,0)$$ 
 (here, $\theta_{E_0,F_0}'$ is the derivative of $\theta_{E_0,F_0}$) will send $\CL^{F_0}(X, F_0)\oplus \CL^{E_0}(X, E_0)$ bijectively onto a complemented subspace of $\CL(X)$. 

As $\CL^{F_0}(X, F_0)\oplus \CL^{E_0}(X, E_0)$ is already a complemented subspace of $\CL(X)$, this claim is established if one can show that $\theta_{E_0,F_0}'(0,0)$ is the inclusion map. 
To see this, we observe that for every $\epsilon \in (0,1)$ and $(R,S)\in \CL^{F_0}(X, F_0)\oplus \CL^{E_0}(X, E_0)$, with $\|R\| + \|S\| < \epsilon$, one has, via \eqref{eqt:theta-E0-F0}, 
\begin{align*}
\| \theta_{E_0,F_0}(R,S) - \theta_{E_0,F_0}(0,0) - (R+S)\| 
& = \big\| RS - SR - RSR\big\|  < \epsilon^2(2+\epsilon).
\end{align*}
This gives $\theta_{E_0,F_0}'(0,0)(R,S) = R+S$, as required.

Finally,  we will establish that $\big(\CI(X), \SG(X), \kappa\big)$ is a locally trivial $\BK$-analytic affine-Banach bundle. 
Indeed, as $\pi_{E_0,F_0} \circ \kappa\circ \mu_{E_0,F_0}^{-1}(R,S) = R$, 
we know that $\kappa: \CI(X)\to \SG(X)$ is an analytic. 
Moreover, since the definition of the analytic atlas as in \eqref{eqt:anal-atlas-IX} is defined via the map $\Theta_{E_0,F_0}$ as well as the bi-analytic maps $\pi_{E_0,F_0}$ and $\phi_{E_0,F_0}$, it is a tautology that $\Theta_{E_0,F_0}$ is bi-analytic; i.e., Condition (B2) of Definition \ref{defn:aff-Ban}(b) holds. 

Suppose that $\varphi$ is the map as in Condition (B3) for $\Theta_{E_1,F_1}^{-1}\circ \Theta_{E_0, F_0}$.
It follows from  \eqref{eqt:Xi-pi-inv} and Lemma \ref{lem:square-zero} that for any $E\in \CC_{F_0}\cap \CC_{F_1}$ and $\idp\in \CI(X)_{E_0}$, one has
\begin{align*}
\varphi(E)(\idp) 
& = \Ad\big(\Xi_{E_1,F_1}(E)^{-1}\big)\circ \Ad\big(\Xi_{E_0,F_0}(E)\big)(\idp) \\
& = \big(I - \pi_{E_1,F_1}(E)\big)\big(I+\pi_{E_0,F_0}(E)\big)\idp\big(I-\pi_{E_0,F_0}(E)\big)\big(I + \pi_{E_1,F_1}(E)\big).
\end{align*}
Let us define $\psi: \CC_{F_0}\cap \CC_{F_1}\to \CL(\CL(X))$ by 
\begin{align*}
\psi(E)(T) 
& = \Ad\big(I - \pi_{E_1,F_1}(E)\big)\circ \Ad\big(I+\pi_{E_0,F_0}(E)\big)(T) \qquad (T\in \CL(X)). 
\end{align*}
Since both $\pi_{E_0,F_0}$ and $\pi_{E_1,F_1}$ are analytic, we know that the map $\psi$ is analytic.
On the other hand, as $\CL^{E_1}(X,E_1)\in \SG(\CL(X))$, we obtain 
a continuous affine map 
$$\Gamma: \CL(\CL(X))\to \CA\big(\CI(X)_{E_0}, \CI(X)_{E_1}\big)$$ 
satisfying the condition in Lemma \ref{lem:comp-affine-subsp}. 
It follows that
$$\Gamma(\psi(E)) = \psi(E)|_{\CI(X)_{E_0}} = \varphi(E) \qquad (E\in \CC_{F_0}\cap \CC_{F_1}).$$ 
Since $\Gamma\circ \psi$ is an analytic map from $\CC_{F_0}\cap \CC_{F_1}$ to $\CA\big(\CI(X)_{E_0}, \CI(X)_{E_1}\big)$, Condition (B3) of Definition \ref{defn:aff-Ban}(b) is satisfied. 
\end{proof}



\medskip

\begin{eg}\label{eg:I(C2)}
We equip $\BC^2$ with the usual Euclidean norm. 
For $\lambda\in \BC$, we set $E_\lambda:=\left\{ \begin{bmatrix}
a \\
\lambda a
\end{bmatrix}:a\in \BC\right\}$
and $E_\infty:= \left\{ \begin{bmatrix}
0 \\
b
\end{bmatrix}:b\in \BC\right\}$.
Then $\SG(\BC^2) = \big\{E_\lambda:\lambda\in \BC\cup \{\infty\} \big\}$ 
and 
$$\CI(\BC^2) = \left\{ \begin{bmatrix}
1 & \gamma\\
0 & 0
\end{bmatrix}:\gamma\in \BC\right\}
\cup 
\left\{ \begin{bmatrix}
1-\alpha & \alpha/\lambda\\
\lambda(1-\alpha) & \alpha
\end{bmatrix}:\alpha, \lambda\in \BC; \lambda\neq 0\right\}
\cup 
\left\{ \begin{bmatrix}
0 & 0\\
\delta & 1
\end{bmatrix}:\delta\in \BC\right\}.$$

Notice also that $E_0\smallT E_\infty$, 
$\idp^{E_\infty}_{E_0} = \begin{bmatrix}
1 & 0\\
0 & 0
\end{bmatrix}$ and 
$\idp^{E_0}_{E_\infty} = \begin{bmatrix}
0 & 0\\
0 & 1
\end{bmatrix}$. 
Moreover, for each $\lambda\in \BC\setminus \{0\}$, we have $E_\lambda \smallT E_\infty$,  $E_\lambda \smallT E_0$, 
$\idp^{E_\infty}_{E_\lambda} = \begin{bmatrix}
1 & 0\\
\lambda & 0
\end{bmatrix}$ 
as well as 
$\idp^{E_0}_{E_\lambda} = \begin{bmatrix}
0 & 1/\lambda\\
0 & 1
\end{bmatrix}$. 

On the other hand, $\CI(\BC^2)_{E_0} = \left\{ \begin{bmatrix}
1 & \gamma\\
0 & 0
\end{bmatrix}:\gamma\in \BC\right\}$, 
$\CI(\BC^2)_{E_\infty} = \left\{ \begin{bmatrix}
0 & 0\\
\delta & 1
\end{bmatrix}:\delta\in \BC\right\}$,  
$$\CL^{E_\infty}(\BC^2, E_\infty) = \left\{ \begin{bmatrix}
0 & 0\\
\delta & 0
\end{bmatrix}:\delta\in \BC\right\}\quad \text{and} 
\quad \CL^{E_0}(\BC^2, E_0) = \left\{ \begin{bmatrix}
0 & \gamma\\
0 & 0
\end{bmatrix}:\gamma\in \BC\right\}.$$ 
Furthermore, for $\lambda \in \BC\setminus \{0\}$, one has $\CI(\BC^2)_{E_\lambda} = \left\{ \begin{bmatrix}
1-\alpha & \alpha/\lambda\\
\lambda(1-\alpha) & \alpha
\end{bmatrix}:\alpha\in \BC\right\}$  and 
$$\CL^{E_\lambda}(\BC, E_\lambda) = \left\{ \alpha\begin{bmatrix}
-1 & 1/\lambda\\
-\lambda & 1
\end{bmatrix}:\alpha\in \BC\right\}.$$

The map $\pi_{E_0, E_\infty}:\CC_{E_\infty}  = \{E_\lambda:\lambda\in \BC \}\to \CL^{E_\infty}(\BC^2, E_\infty)$ is given by
$$\pi_{E_0, E_\infty}\left(E_\lambda\right) =  \begin{bmatrix}
0 & 0\\
\lambda & 0
\end{bmatrix} \qquad (\lambda \in \BC).$$
We also have $\CC_{E_0} = \big\{E_\lambda:\lambda\in \BC\setminus \{0\} \big\}\cup \{E_\infty \}$. 
The map $\pi_{E_\infty, E_0}:\CC_{E_0}\to \CL^{E_0}(\BC^2, E_0)$ is given by
$$\pi_{E_\infty, E_0}\left(E_\infty\right) =  \begin{bmatrix}
0 & 0\\
0 & 0
\end{bmatrix}
\qquad \text{and} \qquad
\pi_{E_\infty, E_0}\left(E_\lambda\right) =  \begin{bmatrix}
0 & 1/\lambda\\
0 & 0
\end{bmatrix} 
\quad (\lambda \in \BC\setminus \{0\}).$$

Consider $\lambda \in \BC$. 
If $\lambda\neq 0$, then $\mu_{E_0,E_\infty}: \kappa^{-1}(E_\lambda)\to \CL^{E_\infty}(\BC^2, E_\infty)\oplus \CL^{E_0}(\BC^2, E_0)$ is given by 
$$\mu_{E_0,E_\infty}\left(\begin{bmatrix}
1-\alpha & \alpha/\lambda\\
\lambda(1-\alpha) & \alpha
\end{bmatrix}\right) = \left(\begin{bmatrix}
0 & 0\\
\lambda & 0
\end{bmatrix},  \begin{bmatrix}
0 & \alpha/\lambda\\
0 & 0
\end{bmatrix}\right).$$
In the case when $\lambda = 0$, the map $\mu_{E_0,E_\infty}: \kappa^{-1}(E_0)\to \CL^{E_\infty}(\BC^2, E_\infty)\oplus \CL^{E_0}(\BC^2, E_0)$ is given by 
$\mu_{E_0,E_\infty}\left(\begin{bmatrix}
1 & \gamma\\
0 & 0
\end{bmatrix}\right) = \left(\begin{bmatrix}
0 & 0\\
0 & 0
\end{bmatrix},  \begin{bmatrix}
0 & \gamma\\
0 & 0
\end{bmatrix}\right).$

On the other extreme, the map $\mu_{E_\infty,E_0}: \kappa^{-1}(\CC_{E_0})\to \CL^{E_0}(\BC^2, E_0)\oplus \CL^{E_\infty}(\BC^2, E_\infty)$ is given by 
$$\mu_{E_\infty,E_0}\left(\begin{bmatrix}
1-\alpha & \alpha/\lambda\\
\lambda(1-\alpha) & \alpha
\end{bmatrix}\right) = \left(\begin{bmatrix}
0 & 1/\lambda\\
0 & 0
\end{bmatrix},  \begin{bmatrix}
0 & 0\\
\lambda(1-\alpha) & 0
\end{bmatrix}\right)$$ 
and
$\mu_{E_\infty, E_0}\left(\begin{bmatrix}
0 & 0\\
\delta & 1
\end{bmatrix}\right) = \left(\begin{bmatrix}
0 & 0\\
0 & 0
\end{bmatrix},  \begin{bmatrix}
0 & 0\\
\delta & 0
\end{bmatrix}\right)$. 

The two charts 
$\big(\kappa^{-1}(\CC_{E_\infty}), \mu_{E_0,E_\infty}, \CL^{E_\infty}(X,E_\infty)\oplus \CL^{E_0}(X, E_0)\big)$ and  
$$\big(\kappa^{-1}(\CC_{E_0}), \mu_{E_\infty,E_0}, \CL^{E_0}(X,E_0)\oplus \CL^{E_\infty}(X, E_\infty)\big)$$
form an analytic atlas for $\CI(\BC^2)$, which produces the Banach submanifold structure induced from $M_2(\BC)$. 
\end{eg}

\medskip

There is another way to consider fibration of $\CI(X)$ over $\SG(X)$, namely, through the map $\kappa': \CI(X)\to \SG(X)$ that sends $\idp$ to $\ker \idp$. 
The same conclusion as in Theorem \ref{thm:main} holds for $(\CI(X),\SG(X), \kappa')$.

\medskip

Let us denote by $\Inv(X)$ the set of ``self-inverse mappings''; i.e. 
$$\Inv(X): = \big\{V\in \GL(X)\setminus \{I, -I\} : V^2 = I\big\}.$$
For each $V\in \Inv(X)$, let us denote $X^V:= \{x\in X: V(x)= x \}$, and set $\bar \kappa: \Inv(X) \to \SG(X)$ to be the map given by $\bar \kappa(V) := X^V$.  
Since the bi-analytic bijection  $T \mapsto 2T - I $ sends $\CI(X)$ onto $\Inv(X)$, Theorem \ref{thm:main} tells us that $\Inv(X)$ is a Banach submanifold of $\CL(X)$. 
More precisely, define $\bar \mu_{E,F}: \bar \kappa^{-1}(\CC_{F}) \to \CL^{F}(X,F)\oplus \CL^{E}(X, E)$ by 
$$\bar \mu_{E,F}(V):= \big( \idp^F_{(I_X+V)(X)} - \idp^F_E, \idp^F_EV\idp^E_F/2\big) \qquad (V\in \Inv(X)_E)$$
(see \eqref{eqt:mu-E-F}). 
Then 
$\bar \mu_{E,F}^{-1}(R,S)= 2(RS+ R-RSR + S -SR) +  \idp^{F}_E - \idp^{E}_F$
(see \eqref{eqt:mu-E-F-inv}), and 
$\big\{\bar \kappa^{-1}(\CC_{F}), \bar \mu_{E,F}, \CL^{F}(X,F)\oplus \CL^{E}(X, E): E,F\in \SG(X); F\smallT E \big\}$
is an analytic atlas for the Banach submanifold structure on $\Inv(X)$ induced from $\CL(X)$.
Furthermore, $(\Inv(X), \SG(X), \bar \kappa)$ is a locally trivial analytic affine-Banach bundle over $\SG(X)$.

\medskip

Another disguised form of $\CI(X)$ is the subspace 
$$\SG(X)\times_\CC \SG(X) :=\{(E,F)\in \SG(X)\times \SG(X): F \smallT E \}$$ 
of $\SG(X)\times \SG(X)$. 
We will say some words about this subspace in the following.

\medskip

\begin{cor}\label{cor:I=GxG}
Suppose that $\mathcal{T}$ is the topology on $\SG(X)\times_\CC \SG(X)$ induced from  the product topology on $\SG(X)\times \SG(X)$. 
There is a $\BK$-Banach manifold structure on $\SG(X)\times_\CC \SG(X)$ compatible with $\mathcal{T}$ such that under the projection $\kappa_1: \SG(X)\times_\CC \SG(X) \to \SG(X)$ onto the first coordinate, one obtain a locally trivial analytic affine-Banach bundle structure on $\SG(X)\times_\CC \SG(X)$.
\end{cor}
\begin{proof}
By Theorem \ref{thm:main}, it suffices to show that $\Psi: (E,F)\mapsto \idp^F_E$ is a homeomorphism from $\SG(X)\times_\CC \SG(X)$ onto $\CI(X)$. 
Moreover, thanks to the continuity of $\kappa$, it suffices to establish the continuity of $\Psi$. 
For this, let us consider a sequence $\{(E_n,F_n)\}_{n\in \BN}$ in $\SG(X)\times_\CC \SG(X)$ converging to $(E_0,F_0)\in \SG(X)\times_\CC \SG(X)$. 
By Corollary \ref{cor:conv-in G(X)}, we may assume that $E_n\in \CC_{F_0}$ and $F_n\in \CC_{E_0}$ for all $n\in \BN$. 
Relation \eqref{eqt:Xi-inv} and Corollary \ref{cor:conv-in G(X)} produce $\idp^{\Xi_{E_0,F_0}(E_n)^{-1}(F_n)}_{E_0}\to \idp^{F_0}_{E_0}$.  
Moreover,  the continuity of $\Theta_{E_0,F_0}$ gives the required convergence:
$\idp^{F_n}_{E_n} = \Theta_{E_0,F_0}\big(E_n, \idp^{\Xi_{E_0,F_0}(E_n)^{-1}(F_n)}_{E_0}\big) \to \idp^{F_0}_{E_0}.$
\end{proof}

\medskip

The structure of the above affine-Banach bundle will be state explicitly in the following. 
For $(E_0, F_0)\in \SG(X)\times_\CC\SG(X)$, we define $\check \mu_{E_0,F_0}: \kappa_1^{-1}(\CC_{F_0})\to \CL^{F_0}(X,F_0)\oplus \CL^{E_0}(X, E_0)$ to be the map  
$$\check \mu_{E_0,F_0}(E,F) := \big(\idp^{F_0}_E - \idp^{F_0}_{E_0},
\idp^{F_0}_{E_0}\idp^{F}_E\idp^{E_0}_{F_0}\big) \qquad (E\in \CC_{F_0}; F\in \CC_E).$$
Then 
$\check\mu_{E_0,F_0}^{-1}(R,S) =\big((I_X+R)\big(S+\idp^{F_0}_{E_0}\big)(X), \ker \big(S+\idp^{F_0}_{E_0}\big)(I_X-R)\big)$
and 
$\big\{\check\kappa^{-1}(\CC_{F}), \check \mu_{E,F}, \CL^{F}(X,F)\oplus \CL^{E}(X, E): (E,F)\in \SG(X)\times_\CC\SG(X)\big\}$
is an analytic atlas for $\SG(X)\times_\CC \SG(X)$. 

\medskip

A direct consequence of Corollary \ref{cor:I=GxG} is the well-known fact that $\SG(\BK^n)\times_\CC \SG(\BK^n)$ is not closed in $\SG(\BK^n)\times \SG(\BK^n)$ for any $n\geq 2$.
In fact, $\CI(\BK^n)$ is never norm compact because it contains non-zero affine-Banach subspaces, but $\SG(\BK^n)\times \SG(\BK^n)$ is compact.


\medskip

\section{$\CI(X)$ as a Banach bundle}

\medskip

It is natural to ask if $(\CI(X), \SG(X), \kappa)$ is actually a Banach bundle, instead of an affine-Banach bundle. 
The first proposition in this section is that one can regard $(\CI(X), \SG(X), \kappa)$ as a continuous Banach bundle. 

\medskip

In fact, Proposition \ref{prop:cont-sect-F-Ban-bundle} tells us that it suffices to show the existence of a continuous global cross section for $(\CI(X), \SG(X), \kappa)$. 
In order to construct such a global cross section, let us fix an element $F_E\in \CC_E$ for every $E\in \SG(X)$. 
Since $\SG(X)$ is metrizable (see e.g., \cite[\S 2.1]{Doua}), there is a partition of unity $\{\psi_E \}_{E\in \SG(X)}$, consisting of continuous functions, dominated by the open covering $\{\CC_{F_E} \}_{E\in \SG(X)}$ of $\SG(X)$.
On the other hand, as $\kappa^{-1}(\CC_{F_E})$ is homeomorphic to a trivial affine-Banach bundle, there exists a local cross section on it. 
Now, a standard ``scaled-sum construction'' will produce the required continuous global cross section (observe that as $\CI(X)_E$ is an affine subspace of $\CL(X)$, it is closed under convex combinations). 

\medskip

Let us state this clearly as follows. 

\medskip

\begin{prop}\label{prop:loc-tri-cont-Ban}
If $X$ is a $\BK$-Banach space, then $(\CI(X), \SG(X), \kappa)$ is a locally trivial continuous $\BK$-Banach bundle, under equivalent $\BK$-Banach space structures on all the fibers of $\kappa$. 
\end{prop}

\medskip

By Proposition \ref{prop:cont-sect-F-Ban-bundle}(b), the only obstruction for $\CI(X)$ to be identified with a locally trivial analytic Banach bundle is the existence of an analytic global cross section. 
However, a complex analytic global cross section does not exist even in the case when $X = \BC^2$. 

\medskip

\begin{eg}
Let $E_\lambda$ be as in Example \ref{eg:I(C2)}. 
Suppose that there is a complex analytic global cross section $\rho: \SG(\BC^2) \to \CI(\BC^2)$. 
Then one can find a function $\alpha:\BC\setminus \{0\} \to \BC$ satisfying 
$$\rho \circ \pi_{E_0,E_\infty}^{-1}\left(\begin{bmatrix}
0 & 0\\
\lambda & 0
\end{bmatrix}\right) = \begin{bmatrix}
1-\alpha(\lambda) & \alpha(\lambda)/\lambda\\
\lambda(1-\alpha(\lambda)) & \alpha(\lambda)
\end{bmatrix} 
\qquad (\lambda\in \BC\setminus \{0\}).$$
As $\rho\circ \circ \pi_{E_0,E_\infty}^{-1}$ is complex analytic, we know that $\alpha$ is holomorphic. 
Moreover, the compactness of $\SG(\BC^2)$ tells us that the image of $\rho$ is norm-bounded. 
From this, we deduce that the three functions 
$$\lambda\mapsto \alpha(\lambda), \quad \lambda \mapsto \alpha(\lambda)/\lambda \quad \text{and} \quad \lambda \mapsto \lambda(1-\alpha(\lambda))$$ 
are bounded on $\BC\setminus \{0\}$. 
As $\alpha$ is bounded, it has a removable singularity at $0$.
Thus, $\alpha$ extends to a bounded entire function on $\BC$, which can only be a constant function. 
On the other hand, since $\lambda \mapsto \alpha(\lambda)/\lambda$ is bounded as well, we know that $\alpha$ is the constant zero function. 
However, this will contradict with the boundedness of $\lambda \mapsto \lambda(1-\alpha(\lambda))$.
Consequently, there does not exist a complex analytic global cross section on $\CI(\BC^2)$. 
In other words, $\big(\CI(\BC^2), \SG(\BC^2\big), \kappa)$ is not a locally trivial complex analytic Banach bundle. 
\end{eg}

\medskip

More generally, there is no complex differentiable global cross section on $(\CI(X), \SG(X), \kappa)$ for any complex Banach space $X$. 

\medskip

\begin{prop}\label{prop:not-complex-anal}
Let $X$ be a complex Banach space. 
We denote by $\SG(X)_\fin$ the subset of $\SG(X)$ consisting of finite dimensional subspaces, and set 
$$\CI(X)_\fin:= \kappa^{-1}(\SG(X)_\fin).$$ 
Then the subbundle $(\CI(X)_\fin, \SG(X)_\fin, \kappa)$ of $(\CI(X), \SG(X), \kappa)$
is not a locally trivial complex differentiable Banach bundle (see Remark \ref{rem:gen-aff-Ban}(b)). 
\end{prop}
\begin{proof}
Suppose on the contrary that $(\CI(X)_\fin, \SG(X)_\fin, \kappa)$
is a locally trivial complex differentiable Banach bundle. 
Then there is a complex differentiable global cross section $\rho: \SG(X)_\fin\to \CI(X)_\fin$. 
Consider $E_0\in \SG(X)_\fin$ and $F_0\in \CC_{E_0}$. 

Fix an operator $T\in \CL(E_0, F_0)$ and put $Y:= E_0 + T(E_0)$. 
Then $Y$ is a finite dimensional subspace. 
For any $\lambda\in \BC$, one has 
$$p_{E_0,F_0}^{-1}(\lambda T) = (I + \lambda T)(E_0)\subseteq Y,$$
where $p_{E_0,F_0}$ is as in \eqref{eqt:defn-p-E-F}. 
Hence, $\{p_{E_0,F_0}^{-1}(\lambda T):\lambda\in \BC \}\subseteq \SG(Y)\subseteq \SG(X)_\fin$. 

Choose any $x\in X$ and any $f$ in the dual space, $X^*$, of $X$. 
As the inclusion map $\iota:\CI(X)\to \CL(X)$ is complex analytic, we know that the function $\chi$ on $\BC$ defined by 
$$\chi(\lambda) := f\big(\rho\big(p_{E_0,F_0}^{-1}(\lambda T)\big)(x)\big)$$ 
is holomorphic (since \eqref{eqt:atlas-Upm} is an analytic atlas). 
On the other hand, as $\SG(Y)$ is compact and $\rho|_{\SG(Y)}$ is continuous, we know that 
$${\sup}_{E\in \SG(Y)} \|\rho(E)\| < \infty. $$
This means that $\chi$ is bounded and hence it should be a constant function. 

Since $x$ and $f$ are arbitrarily chosen, we know that $\lambda \mapsto \rho\big(p_{E_0,F_0}^{-1}(\lambda T)\big)$ is a constant map, for each fixed $T\in \CL(E_0,F_0)$. 
Consequently, the map $\rho\circ p_{E_0,F_0}^{-1}:\CL(E_0, F_0)\to \CI(X)$ is constant (since all the rays pass through $0$), 
but this contradicts with the fact that $\rho$ is a cross section on $\CC_{F_0}$. 
\end{proof}

\medskip

In the following, we consider the case of (real or complex) Hilbert spaces.  

\medskip

\begin{prop}\label{prop:Hil-sp-Ban-bun}
If $H$ is a complex Hilbert space, then $(\CI(H), \SG(H), \kappa)$ is a locally trivial real analytic complex Banach bundle, under equivalent complex Banach space structures on the fibers such that the self-adjoint projection corresponding to the base point of a fiber is the zero element of the vector space structure on that fiber. 
\end{prop}
\begin{proof}
For any $E\in \SG(H)$, we denote by $E^\bot$ the orthogonal complement of $E$, and set 
$$\prj_E:=\idp^{E^\bot}_E.$$ 
By the argument of Proposition \ref{prop:cont-sect-F-Ban-bundle}(b), the conclusion is obtained if one can show that $E\mapsto \prj_E$ 
is a real analytic map from $\SG(H)$ to $\CI(H)$, or equivalently from $\SG(H)$ to $\CL(H)$ (because $\CI(H)$ is a Banach submanifold of $\CL(H)$). 
However, this fact was already proved in \cite[Proposition 4(4)]{AM} (see also \cite[Proposition 4(1)]{AM}). 
\end{proof}

\medskip

Note that $E \mapsto \prj_{E}$ is never a complex analytic map from $\SG(H)$ to $\CL(H)$, because of the proof of Proposition \ref{prop:not-complex-anal}.

\medskip

When $K$ is  a real Hilbert space, one has the stronger conclusion that $(\CI(K), \SG(K), \kappa)$ can be identified with the tangent bundle of $\SG(K)$.
In fact, the total space $\mathbf{T}(\SG(K))$ of the tangent bundle  of $\SG(K)$ is the disjoint union ${\biguplus}_{E\in \SG(K)}\CL(E, E^\bot)$ of Banach spaces, equipped with an appropriate Banach manifold structure.  
By Lemma \ref{lem:I-F-affine}, for every $E\in \SG(K)$ and $T\in \CL(E, E^\bot)$, one knows that $T^*\bar \prj_{E^\bot} + \prj_{E}$ is in $\CI(X)$  (where $T^*$ is the adjoint of $T$), where $\bar \prj_{E^\bot}$ is the orthogonal projection $\prj_{E^\bot}$ regarding as a map from $K$ to $E^\bot$. 

\medskip

\begin{thm}\label{thm:Hil-tang}
Let $K$ be a real Hilbert space. 
The assignment $(E,T)\mapsto T^*\bar \prj_{E^\bot} + \prj_{E}$ induces a fiberwise affine bi-analytic bijection from $\mathbf{T}(\SG(K))$ onto $\CI(K)$, when $\CI(K)$ is equipped with the Banach submanifold structure induced from $\CL(K)$. 
\end{thm}
\begin{proof}
For any $E\in \SG(K)$ and $F\in \CC_E$, we can identify, via the Banach space isomorphism $\Lambda_{F,E}$ as in \eqref{eqt:defn:L-F0-X-F0}, 
\begin{equation}\label{eqt:id-LE}
\CL(F, E)\cong \CL^{E}(K,E);
\end{equation}
in this case, $T\in \CL(F,E)$ is identified with $T\bar \idp^E_F$ (note that $\bar\idp^E_{E^\bot} = \bar \prj_{E^\bot}$).
We may sometime regard $T\in \CL(F,E)$ as an operator from $F$ to $K$ as well. 

Let us set
$$\hat \CI(K):={\biguplus}_{E\in \SG(K)}\CL(E^\bot, E)$$
and denote by $\hat \kappa:\hat \CI(K)\to \SG(K)$ the map that sends $T\in \CL(E^\bot, E)$ to $E$. 

Through identification in \eqref{eqt:id-LE}, together with the equality 
$$\CL^E(K, E)=\CI(H)_E-\prj_E\qquad (E\in \SG(K)),$$ 
one may equate $(\CI (K), \SG(K), \kappa)$ with $(\hat \CI(K), \SG(K), \hat\kappa)$, and obtains a Banach manifold structure on $\hat \CI(K)$.
In this case, the following is an analytic atlas for this structure on $\hat \CI(K)$: 
$$\big\{\hat \kappa^{-1}(\CC_{E_0^\bot}), \hat\mu_{E_0}, \CL^{E_0^\bot}(K,E_0^\bot)\oplus \CL^{E_0}(K,E_0): E_0\in \SG(K)\big\},$$ 
where
\begin{equation}\label{eqt:defn-ti-mu}
\hat \mu_{E_0}(T): = \big(\idp^{E_0^\bot}_{E} - \prj_{E_0},  \prj_{E_0}(T\circ \prj_{E^\bot} +\prj_{E})\prj_{E_0^\bot}\big) \qquad \big(E\in \CC_{E_0^\bot}, T\in \CL(E^\bot,E)\big)
\end{equation}
(c.f. \eqref{eqt:mu-E-F}). 
Moreover, for $R\in \CL^{E_0^\bot}(K,E_0^\bot)$ and $S\in \CL^{E_0}(K,E_0)$, we have, via \eqref{eqt:mu-E-F-inv-1},
\begin{equation}\label{eqt:hat-mu-inv}
\hat \mu_{E_0}^{-1}(R,S) = (I+R)\big(S + \prj_{E_0}\big)(I-R)|_{(I+R)(E_0)^\bot}.
\end{equation}
In order to verify this proposition, it suffices to show that the fiberwise linear map 
$$\Phi: \hat \CI(K)\to \mathbf{T}(\SG(K))$$ 
that sends an operator to its adjoint is bi-analytic.  

To do this, let us fix $E_0\in \SG(K)$. 
Consider $E_1\in \CC_{E_0^\bot}$ and $T\in \CL(E_1,E_1^\bot)$.
We define 
$$\gamma(s):= p_{E_1,E_1^\bot}^{-1}(s T)\qquad (s\in (-1,1)).$$
We also define
\begin{equation}\label{eqt:def-delta}
\nu_{E_0}(T) := \big(\pi_{E_0,E_0^\bot}(E_1), \Psi(T)\bar \prj_{E_0}\big),
\end{equation}
where $\Psi(T) := (p_{E_0,E_0^\bot}\circ \gamma)'(0)$; i.e. the derivative of $p_{E_0,E_0^\bot}\circ \gamma$ at $0$. 
As in \cite[p.69 \& p.73]{Chu12}, 
$$\Big({\biguplus}_{E\in \CC_{E_0^\bot}}\CL(E, E^\bot), \nu_{E_0}, \CL^{E_0^\bot}(K,E_0^\bot)\times \CL^{E_0^\bot}(K,E_0^\bot) \Big)$$ 
is a local chart of $\mathbf{T}(\SG(K))$ around $E_0$.

Now, we are required to know the map $p_{E_0,E_0^\bot}\circ p_{E_1,E_1^\bot}^{-1}$, in order to express $\nu_{E_0}(T)$.
For this, we need to express the two maps $p_{E_0,E_0^\bot} \circ p_{E_1,E_0^\bot}^{-1}$ and $p_{E_1,E_0^\bot} \circ p_{E_1,E_1^\bot}^{-1}$.
Let us  put $B:= p_{E_1,E_0^\bot}(E_0)$ 
(i.e. $(I+B)(E_1) = E_0$) 
and  $C:=p_{E_1^\bot, E_1}(E_0^\bot)$  
(i.e. $(I+C)(E_1^\bot) = E_0^\bot$). 
Denote 
$$D:= I_{E_1} + B \quad \text{and} \quad A:= I_{E_1^\bot} + C,$$ 
where $I_{E_1}$ and $I_{E_1^\bot}$ are the identity maps on $E_1$ and $E_1^\bot$, respectively. 
Then $D = \big(I + \pi_{E_1,E_0^\bot}(E_0)\big)|_{E_1}$ is an operator from $E_1$ to $E_0$ and $A = \big(I + \pi_{E_1^\bot, E_1}(E_0^\bot)\big)|_{E_1^\bot}$ is an operator from $E_1^\bot$ to $E_0^\bot$. 
As in  the proof of \cite[Lemma 3.12]{Upm}, for any $S\in \CL(E_1, E_1^\bot)$, one has 
\begin{align*}
p_{E_0,E_0^\bot} \circ p_{E_1,E_0^\bot}^{-1} \circ p_{E_1,E_0^\bot} \circ p_{E_1,E_1^\bot}^{-1} (S) = \big(A S(I_{E_1} - CS)^{-1} - B\big)D^{-1}
\end{align*}
(note that in \cite{Upm}, the notation $a,b,c$ and $d$ were used, instead of $A,B,C$ and $D$).  
Hence, for $s\in (-1,1)$, we have  
$$(p_{E_0,E_0^\bot}\circ \gamma)'(s) = AT(I_{E_1} - sC T)^{-1}D^{-1} + sAT(I_{E_1} - sC T)^{-2}CTD^{-1},$$ 
which gives 
\begin{equation}\label{eqt:form-Psi}
\Psi(T) = A T D^{-1}  = (I + C\bar P_{E_1^\bot})TD^{-1}.
\end{equation}

Set $R:= \pi_{E_0,E_0^\bot}(E_1)$. 
It follows from the equality
$B\bar\idp^{E_0^\bot}_{E_1} = \pi_{E_1,E_0^\bot}(E_0)$
as well as Relation \eqref{eqt:converse-map} that $R = - B\bar\idp^{E_0^\bot}_{E_1}$. 
This means that 
\begin{equation}\label{eqt:DE}
D = (I - R)|_{E_1}, 
\end{equation}
and Lemma \ref{lem:square-zero} gives 
$D^{-1} = (I + R)|_{E_0}$.  
On the other hand, as $R= \idp^{E_0^\bot}_{E_1} - \prj_{E_0}$, we have 
\begin{align}\label{eqt:CE-circ-P}
C\bar \prj_{E_1^\bot}
& = \ \pi_{E_1^\bot,E_1}(E_0^\bot)
\ = \ \idp^{E_1}_{E_0^\bot} - \prj_{E_1^\bot}
\ = \ \prj_{E_1} - R - \prj_{E_0}.
\end{align}
From these, we conclude that 
\begin{equation}\label{eqt:delta-E0}
\nu_{E_0}(T)  = \Big(\pi_{E_0,E_0^\bot}(E_1), \big(\prj_{E_1} - \pi_{E_0,E_0^\bot}(E_1) + \prj_{E_0^\bot}\big)T \bar P_{E_1}\big(I+\pi_{E_0,E_0^\bot}(E_1)\big)\prj_{E_0}\Big).
\end{equation}

Now, the adjoint map $\Phi$ restricts to a fiberwise linear bijection 
$$\Phi_{E_0}: {\biguplus}_{E\in \CC_{E_0^\bot}}\CL(E^\bot, E) \to {\biguplus}_{E\in \CC_{E_0^\bot}}\CL(E, E^\bot).$$ 
Under the corresponding local charts of $\CI(K)$ and $\mathbf{T}(\SG(K))$ about $\CC_{E_0^\bot}$, the map $\Phi_{E_0}$ is transformed into a map that sends $(R,S)\in \CL^{E_0^\bot}(K,E_0^\bot)\oplus \CL^{E_0}(K,E_0)$ to 
$$\Big(R,(\prj_{E_0^\bot} + \prj_{(I+R)(E_0)} - R)(I-R^*)\big(S^* + \prj_{E_0}\big)(I+R^*)(I+R)\prj_{E_0}\Big)$$ 
(see  Relations \eqref{eqt:hat-mu-inv} and \eqref{eqt:delta-E0}).
This shows that $\Phi_{E_0}$ is real analytic, because $(I+R)(E_0)=\pi_{E_0,E_0^\bot}^{-1}(R)$ and the assignment $E\mapsto \prj_E$ is a real analytic map  (note that \cite[Proposition 4(4)]{AM} is also valid in the real case using the same argument). 

Conversely, consider again $E_0\in \SG(K)$ and $E_1\in \CC_{E_0^\bot}$. 
Suppose that $R\in \CL^{E_0^\bot}(K,E_0^\bot)$ satisfying $E_1= \pi_{E_0,E_0^\bot}^{-1}(R)$, and $S\in \CL^{E_0^\bot}(K,E_0^\bot)$.   
Let $A,B,C$ and $D$ be the operators as in the above.  
If $T\in \CL(E_1,E_1^\bot)$ such that $\Psi(T)\bar P_{E_0} = S$, then it follows from \eqref{eqt:form-Psi} that 
$$T = A^{-1}S|_{E_0}D.$$
Hence, we have $\nu_{E_0}^{-1}(R,S) = A^{-1}S|_{E_0}D$ (see \eqref{eqt:def-delta}). 

Since $A = \big(I + \pi_{E_1^\bot,E_1}(E_0^\bot)\big)|_{E_1^\bot}$, it follows from Lemma \ref{lem:square-zero} that $A^{-1} = \big(I - \pi_{E_1^\bot,E_1}(E_0^\bot)\big)|_{E_0^\bot}$. 
Consequently, Relations \eqref{eqt:DE} and \eqref{eqt:CE-circ-P} 
(notice that $C\bar \prj_{E_1^\bot} = \pi_{E_1^\bot,E_1}(E_0^\bot)$) 
imply
$$\nu_{E_0}^{-1}(R,S) = (I+ R + \prj_{E_0} - \prj_{E_1})S(I - R)|_{E_1}.$$
This, together with Relation \eqref{eqt:defn-ti-mu}, tells us that $\Phi_{E_0}^{-1}$ is transformed (under the corresponding local charts of $\CI(K)$ and $\mathbf{T}(\SG(K))$ near $\CC_{E_0^\bot}$) into a map sending $(R,S)\in \CL^{E_0^\bot}(K,E_0^\bot)\oplus \CL^{E_0^\bot}(K,E_0^\bot)$ to 
$$\Big(R, \prj_{E_0}\big(\prj_{E_1}(I - R^*)S^*(I + R^* + \prj_{E_0} -\prj_{E_1})\prj_{E_1^\bot} + \prj_{E_1}\big)\prj_{E_0^\bot} \Big).$$
Thus, $\Phi_{E_0}^{-1}$ is also real analytic. 
This completes the proof. 
\end{proof}

\medskip

We may also identify the tangent bundle of $\SG(K)$ with either $\Inv(K)$ or $\SG(K)\times_\CC \SG(K)$ (see the discussion following Theorem \ref{thm:main}). 
On the other hand, one can identify $\SG(K)$ with the real Banach submanifold $\Inv_\mathrm{sa}:= \{V\in \Inv(K): V^*=V \}\setminus \{I,-I\}$.

\medskip

Note that the corresponding statement of Theorem \ref{thm:Hil-tang} for general Banach spaces is in general false because $\CL(E,F)$ may not be isomorphic to $\CL(F,E)$ for $E,F\in \SG(X)$ with $E \smallT F$. 

\medskip

In the following, we will consider the case when $H$ is a $\BK$-Hilbert space. 
Let us define a map $\tau: \CI(H) \to \SG(H)\times \CL(H)$ by 
$$\tau(\idp) := (\kappa(\idp),\iota(\idp)) \qquad (\idp\in \CI(H)),$$ 
where $\iota: \CI(H)\to \CL(H)$ is the inclusion map.
Clearly, $\tau$ is a homeomorphism onto its image.  
Moreover, as in the proof of Theorem \ref{thm:main}, $\tau$ is an analytic immersion. 

\medskip

We set 
\begin{equation*}
\ti \CI(H):={\biguplus}_{E\in \SG(H)}\CL^E(H, E),
\end{equation*}
and define the bundle map $\ti \kappa : \ti \CI(H) \to \SG(H)$ canonically. 
One can see from the proofs of Proposition \ref{prop:Hil-sp-Ban-bun} and Theorem \ref{thm:Hil-tang}  that $(\ti \CI(H), \SG(H), \ti \kappa)$ is a locally trivial real analytic Banach bundle. 

\medskip

Consider  a ``$\CL(H)$-valued metric'' on $\ti \CI(H)$ given by 
\begin{equation}\label{eqt:def-oper-Riem-met}
\la S, T\ra_{\CL(H)} := ST^*, \quad \text{for any }S,T\in \ti \CI(H) \text{ with } \ti \kappa(S)=\ti \kappa(T).
\end{equation}
Notice that $\la \cdot, \cdot \ra_{\CL(H)}$ satisfies all the requirements of a metric (i.e. fiberwise inner product) except that it takes values in $\CL(H)$ instead of the scalar field. 
Note also that  
\begin{equation}\label{eqt:range-metric}
\la S, T\ra_{\CL(H)} \in \prj_{\ti \kappa(S)}\CL(H)\prj_{\ti \kappa(T)}.
\end{equation}
For each $x\in H$, we can also define 
\begin{equation*}\label{eqt:pseduo-met}
\la S, T\ra_x := \la ST^*x, x\ra_H. 
\end{equation*}
Then $\{\la \cdot, \cdot\ra_x \}_{x\in H}$ is a family of pseudo-metric on $(\ti \CI(H), \SG(X), \ti \kappa)$ which is \emph{separating}; in the sense that $S = 0$ whenever we have $\la S, S\ra_x = 0$ for all $x\in H$.  
It is easy to see, via the fiberwise linear analytic immersion induced by $\tau$, that $\la \cdot, \cdot\ra_x$ is real analytic. 
On the other hand, if $H$ is separable and $\{x_n\}_{n\in \BN}$ is a countable dense subset of $H$, then the sequence of pseudo-metric $\{\la \cdot, \cdot\ra_{x_n} \}_{n\in \BN}$ is also separating.

\medskip

Let $\SG(H)_\fin$ be the subset of $\SG(H)$ consisting of finite dimensional subspaces, and we put 
$$\ti \CI(H)_\fin:= \ti \kappa^{-1}(\SG(H)_\fin).$$ 
Then $\big(\ti \CI(H)_\fin, \SG(H)_\fin, \ti \kappa\big)$ is a locally trivial real analytic Banach bundle. 
For any $S,T\in \ti\CI(H)_\fin$ with $\ti \kappa(S) = \ti \kappa(T)$, the operator $\la S, T\ra_{\CL(H)}$ is of finite rank (see \eqref{eqt:range-metric}). 
Therefore, we can define a metric $\la \cdot, \cdot\ra_\fin$ on $\ti\CI(H)_\fin$ by  
\begin{equation}\label{eqt:def-metric-fin}
\la S, T\ra_\fin := \mathrm{Tr}(ST^*), \quad \text{for any }S,T\in \ti \CI(H)_\fin \text{ with } \ti \kappa(S)=\ti \kappa(T),
\end{equation}
where $\mathrm{Tr}$ is the canonical densely defined trace on $\CL(H)$. 

\medskip

In the particular case when $K$ is a real Hilbert space, we see from the identification as in \eqref{eqt:id-LE} that $\ti \CI(K)$ is the same as $\hat \CI(K)$. 
Hence, the above discussion, together with Theorem \ref{thm:Hil-tang}, gives the following (notice that both $(F, S)\mapsto (F, S+\prj_F)$ and $(F, R)\mapsto (F, R^*)$ are real bi-analytic maps from $\SG(K)\times \CL(K)$ to itself). 

\medskip

\begin{cor}\label{cor:tan-embed-trivial-bundle}
Let $K$ be a real Hilbert space. 

\smnoind
(a) The assignment $(E,T)\mapsto (E, T\bar \prj_{E})$ is analytic immersion from $\mathbf{T}(\SG(K))$ to the trivial Banach bundle $\big(\SG(K)\times \CL(K), \SG(K), \kappa_0\big)$ which is fiberwise linear.  
This immersion is a homeomorphism from $\mathbf{T}(\SG(K))$ onto the Banach subbundle $\big\{(E,S): E\in \SG(K);  S\in \CL^{E^\bot}(K,E^\bot)\big\}$. 

\smnoind
(b) The $\CL(K)$-valued metric $\la \cdot, \cdot \ra_{\CL(K)}$ as in \eqref{eqt:def-oper-Riem-met} induces a separating family of real analytic pseudo-metrics on $\mathbf{T}(\SG(K))$. 
If $K$ is separable, then one can find a countable separating family of real analytic pseudo-metrics on $\mathbf{T}(\SG(K))$. 

\smnoind
(c) The metric $\la \cdot, \cdot\ra_\fin$ as in \eqref{eqt:def-metric-fin} is a real analytic metric on $\mathbf{T}(\SG(K)_\fin)$.
\end{cor}

\medskip

Note that since all the tangent spaces of $\SG(K)_\fin$ are isomorphic to Hilbert spaces, it is already known that a real analytic metric exists on $\mathbf{T}(\SG(K)_\fin)$. 
The above gives an explicit construction of such a metric.

\medskip

Notice also that the fiber of the Banach bundle $\hat \CI(K)$ over $E\in \SG(K)$ is a Hilbert $\CL(E^\bot)$-module with inner product $\la S, T\ra_0 := S^*T$ (in the usual convention, a Hilbert $C^*$-module is a right module and the operator-valued inner product is conjugate linear in the first variable). 
One may use it to define another family of pseudo-metrics on $\mathbf{T}(\SG(K))$.

\medskip

\section{$\CI(X)$ as a disjoint union of homogeneous spaces}

\medskip

In the case when $X$ is finite dimensional, it is well-known that $\SG(X)$ can be identified with a disjoint union of quotients of $\GL(X)$ by closed Lie subgroups. 
The corresponding fact for $\CI(X)$ may also be known. 
We are going to look at the infinite dimensional situation. 

\medskip

In the following, we will consider the analytic actions $\Ad$ and $\alpha$ of $\GL(X)$ on $\CI(X)$ and $\SG(X)$, respectively (see \eqref{eqt:defn-Ad-act} and \eqref{eqt:def-alpha}). 
We have already constructed in the proof of Theorem \ref{thm:main} an analytic local right inverse $\Xi_{E_0,F_0}$ (see \eqref{eqt:Xi-pi-inv}) for the evaluation map at $E_0$ from $\GL(X)$ to the orbit $\alpha(\GL(X), E_0)$, for every $E_0\in \SG(X)$.  
The following lemma gives  an analytic local right inverse for $\Ad$. 

\medskip

\begin{lem}\label{lem:right-inverse-Ad}
For $(E_0,F_0)\in \SG(X)\times_\CC \SG(X)$, the map from $\kappa^{-1}(\CC_{F_0})$ to $\GL(X)$ defined by 
$$\check \Xi_{E_0,F_0}(\idp^F_E) := \big(\idp^E_F+ \idp^{F_0}_E\big)\big(\idp^{F_0}_E+ \idp^{E_0}_{F_0}\big) \qquad (\idp^F_E\in \kappa^{-1}(\CC_{F_0}))$$
is an analytic local right inverse 
for the evaluation map at $\idp^{F_0}_{E_0}$ from $\GL(X)$ onto the orbit of $\idp^{F_0}_{E_0}$ under the action $\Ad$.
\end{lem}
\begin{proof}
Consider $\idp^F_E\in \kappa^{-1}(\CC_{F_0})$.
We know that both $\idp^{F_0}_E+ \idp^{E_0}_{F_0}$ and $\idp^E_F+ \idp^{F_0}_E$ are invertible (see Lemmas \ref{lem:square-zero} and \ref{lem:I-F-affine}). 
Moreover, one easily check that 
$$\check \Xi_{E_0,F_0}(\idp^F_E)(E_0) = E \quad \text{and} \quad \check \Xi_{E_0,F_0}(\idp^F_E)(F_0)=F.$$ 
This implies that $\check \Xi_{E_0,F_0}: \kappa^{-1}(\CC_{F_0}) \to \GL(X)$ is a local right inverse for the evaluation map. 

Suppose that $(R,S)=\mu_{E_0,F_0}(\idp^F_E)$. 
It follows from \eqref{eqt:mu-E-F} and \eqref{eqt:mu-E-F-inv-1} that 
$$\idp^F_E = (I+R)(S+\idp^{F_0}_{E_0})(I-R)$$ 
and $R = \idp^{F_0}_E + \idp^{E_0}_{F_0} - I$. 
Consequently, 
$$\check \Xi_{E_0,F_0}\big(\mu_{E_0,F_0}^{-1}(R,S)\big) = \big(I - (I+R)(S+\idp^{F_0}_{E_0})(I-R)+R + \idp^{F_0}_{E_0}\big)(R+I),$$
which is analytic as required.  	
\end{proof}

\medskip

The existences of local analytic right inverses for the evaluation maps from $\GL(X)$ to the orbits of $\alpha$ and $\Ad$, respectively, imply that these two actions are locally transitive in the sense of \cite[Definition 8.20]{Upm}.
For any $\idp\in \CI(X)$ as well as $E\in \SG(X)$, let us set 
$$\GL(X)^\idp:= \{ \ivt\in \GL(X): \ivt \idp \ivt^{-1} = \idp\}\quad 
\text{and}\quad 
\GL(X)^E:= \{ \ivt\in \GL(X): \ivt E = E\}.$$ 
\cite[Proposition 8.21]{Upm} produces the following result.

\medskip

\begin{prop}\label{prop:id-homog}
Let $X$ be a $\BK$-Banach space. 
	Suppose that $\idp\in \CI(X)$ and $E\in \SG(X)$. 
Then both $\GL(X)^\idp$ and $\GL(X)^E$ are analytic $\BK$-Banach Lie subgroups of $\GL(X)$. 
The orbits $\Ad(\GL(X),\idp)$ and $\alpha(\GL(X), E)$ are clopen in $\CI(X)$ and $\SG(X)$, respectively. 
Moreover, the canonical bijections from, respectively, $\GL(X)/\GL(X)^\idp$ and  $\GL(X)/\GL(X)^E$ onto $\Ad(\GL(X),\idp)$  and $\alpha(\GL(X), E)$ are $\BK$-bi-analytic. 
\end{prop}

\medskip

For any $W\in \GL(X)$, we denote by $[W]_{\idp}$ and $[W]_{E}$ the images of $W$ in  $\GL(X)/\GL(X)^\idp$ and $\GL(X)/\GL(X)^E$, respectively.  

\medskip

\begin{cor}\label{cor:I(X)-union-homog-sp}
(a) The $\BK$-Banach submanifold $\CI(X)$ of $\CL(X)$ can be identified with a disjoint union of homogeneous spaces of the form $\GL(X)/\GL(X)^\idp$ for some $\idp\in \CI(X)$, via the map $\Sigma_\idp: \GL(X)/\GL(X)^\idp \to \CI(X)$ given by $\Sigma_\idp\big([W]_{\idp}\big) := W\idp W^{-1}$. 

\smnoind
(b) The $\BK$-Banach manifold $\SG(X)$ can be identified with a disjoint union of homogeneous spaces of the form $\GL(X)/\GL(X)^E$ for some $E\in \SG(X)$, via the map $\Sigma^E: \GL(X)/\GL(X)^E \to \SG(X)$ given by $\Sigma^E\big([W]_{E}\big) := WE$. 

\smnoind
(c) Let $\idp\in \CI(X)$. 
The assignment $\upsilon_0: [W]_{\idp}\mapsto [W]_{\idp(X)}$ is a well-defined $\BK$-analytic map from $\GL(X)/\GL(X)^\idp$ to $\GL(X)/\GL(X)^{\idp(X)}$ such that $\kappa\circ \Sigma_Q = \Sigma^{\idp(X)}\circ \upsilon_0$. 
\end{cor}

\medskip

Our next question concerns with connected components of $\CI(X)$. 
Note that when $\GL(X)$ is connected, Proposition \ref{prop:id-homog} tells us that subsets of the form $\Ad(\GL(X),\idp)$ and $\alpha(\GL(X), E)$ are all the components of $\CI(X)$ and $\SG(X)$, respectively.
In this case, if we define an equivalence relation $\sim$ on $\SG(X)\times_\CC \SG(X)$ such that $(E_1,F_1) \sim (E_2,F_2)$ if and only if $E_1$ and $F_1$ are Banach space isomorphic to $E_2$ and $F_2$, respectively, then all the disjoint components of $\CI(X)$ are of the form 
$$\big\{\idp^F_E: (E,F)\in \SG(X)\times_\CC \SG(X) \text{ with } (E,F)\sim (E_0,F_0) \big\}$$ 
for some  $(E_0,F_0)\in \SG(X)\times_\CC \SG(X)$. 

\medskip

In the case when $H$ is a Hilbert space, Kuiper's theorem tells us that $\GL(H)$ is connected (see \cite{Kui}; see also \cite{Ill} for the case when the Hilbert space is non-separable). 
Therefore, one can determine connected components of both $\CI(H)$ and $\SG(H)$ through the dimensions of subspaces and those of their orthogonal complements. 
\medskip

However, $\GL(X)$ is in general not connected. 
Nonetheless, the above still holds for finite dimensional subspaces.
We will establish this fact in our next result.  

\medskip

\begin{prop}\label{prop:fin-dim-conn}
Suppose that $X$ is an infinite dimensional $\BK$-Banach space and $n\in \BN$. 
Then $\CI_n(X):= \{\idp\in \CI(X): \dim \idp(X) = n \}$ and $\SG_n(X):= \{E\in \SG(X):  \dim E =n \}$ are connected component of $\CI(X)$ and $\SG(X)$, respectively. 
\end{prop}
\begin{proof}
In the following, for any $F, F_1, F_2\in \SG(X)$,  we will write $F = F_1\oplus F_2$ if $F_1\cap F_2 = (0)$ and $F = F_1 + F_2$. 

By Proposition \ref{prop:id-homog}, it suffices to show that $\CI_n(X)$ and $\SG_n(X)$ are connected.  
Moreover, since $\kappa$ is a continuous surjection from $\CI_n(X)$ onto $\SG_n(X)$, we only need to establish the path connectedness of $\CI_n(X)$. 
Fix $E_1, E_2\in \SG_n(X)$ as well as $F_1\in \CC_{E_1}$ and $F_2\in \CC_{E_2}$. 
Notice that $E_1+E_2$ is finite dimensional, and we fixed a subspace $F\in \CC_{E_1+E_2}$. 
There exist finite dimensional subspaces $\tilde E_1$ and $\tilde E_2$ such that 
$$E_1 \oplus \tilde E_2 = E_1 + E_2 = \tilde E_1\oplus E_2.$$ 
We know from Lemma \ref{lem:I-F-affine} that $\idp^{F_1}_{E_1}$ is joined to $\idp_1:=\idp^{F\oplus \tilde E_2}_{E_1}$ by a continuous path in $\CI(X)_{E_1}$. 
Similarly, $\idp^{F_2}_{E_2}$ is joined by a continuous path to $\idp_2:=\idp^{F\oplus \tilde E_1}_{E_2}$. 
Therefore, in order to show that $\idp^{F_1}_{E_1}$ and $\idp^{F_2}_{E_2}$ are in the same connected component, we only need to construct a continuous path joining $\idp_1$ and $\idp_2$. 

Pick a subspace $E_3$ of $F$ with dimension $n$  and choose $\check F\in \SG(F)$ with $F = E_3\oplus \check F$. 
Set $\idp_0:= \idp^{\check F \oplus (E_1+E_2)}_{E_3}$. 
Under the decomposition $X = E_3\oplus (\check F \oplus \tilde E_2) \oplus E_1$, one can write
$$\idp_0 = \begin{bmatrix}
I_{E_3} & 0 & 0\\
0 & 0 & 0\\
0 & 0 & 0
\end{bmatrix}
\quad \text{ and } \quad
\idp_1 = \begin{bmatrix}
0 & 0 & 0\\
0 & 0 & 0\\
0 & 0 & I_{E_1}
\end{bmatrix}.$$
Since $\dim E_1 = \dim E_3 = n$, there is an isomorphism $\Phi:E_1\to E_3$. 
For $t\in [0,\pi/2]$, we set 
$$W_t = \begin{bmatrix}
\cos t\cdot I_{E_3} & 0 & \sin t\cdot \Phi\\
0 & I_{\check F \oplus\tilde E_2 } & 0\\
-\sin t\cdot \Phi^{-1} & 0 & \cos t\cdot I_{E_1}
\end{bmatrix}\in \CL(X).$$
It is not hard to check that $W_t\in \GL(X)$. 
Moreover, as $W_0 \idp_0 W_0^{-1} = \idp_0$ and $W_{\pi/2}\idp_0 W_{\pi/2}^{-1} = \idp_1$, we see that $\idp_0$ and $\idp_1$ are joined by a continuous path of idempotents. 

In the same way, $\idp_0$ and $\idp_2$ are joined by a continuous path of idempotents, via the presentations of $\idp_0$ and $\idp_2$ under the decomposition $X = E_3\oplus (\check F \oplus \tilde E_1) \oplus E_2$. 
This completes the proof of the connectedness of $\CI_n(X)$. 
\end{proof}

\medskip 

In the remainder of this section, we will again look at the case of a $\BK$-Hilbert space $H$. 
For any $Q\in \CI(H)$, one can find $P\in \CI(H)$ and $W\in \GL(H)$ such that $P^* = P$ and $Q = WPW^{-1}$ (see e.g. \cite[Proposition 4.6.2]{BKD}). 
Hence, $\Ad\big(\GL(H), P\big) = \Ad\big(\GL(H), Q\big)$. 
This means that in order to study components of $\CI(H)$ it suffices to consider component generated by self-adjoint projections. 

\medskip

Let us denote $\CU(H):= \{V\in \CL(H): VV^* = I = V^*V \}$, and 
$$\CU(H)^P:= \CU(H)\cap \GL(H)^P.$$ 
It is well-known that $\CU(H)$ is a real analytic Banach manifold, and the assignment $V\mapsto VP(H)$ induces a real bi-analytic map 
from 
$\CU(H)/\CU(H)^P$ onto a clopen subset of $\SG(H)$ (see e.g. \cite[Proposition 3(3)]{AM}).

\medskip

Assume that $k:= \dim_\BK P(H) < \infty$. 
Then the image of $\CU(H)/\CU(H)^P$ in $\SG(H)$ is precisely the subset $\SG_k(H)$ of $k$-dimensional subspaces. 
Moreover, Corollary \ref{cor:I(X)-union-homog-sp}(b) and Proposition \ref{prop:fin-dim-conn} gives
$$\GL(H)/\GL(H)^{P(H)} \approx \SG_{k}(H).$$
By Proposition \ref{prop:Hil-sp-Ban-bun} as well as Corollary \ref{cor:I(X)-union-homog-sp}(c), one obtains a locally trivial real analytic $\BK$-Banach bundle 
$\big(\GL(H)/\GL(H)^P, \CU(H)/\CU(H)^P, \upsilon\big)$. 

\medskip

It is natural to ask if there is an explicit way to express this map $\upsilon$ (which is defined through several identifications). 
%
%
In the following, we will describe it via the Gram-Schmidt process.
Consider $W\in \GL(H)$. 
We are required to find $V\in \CU(H)$ satisfying 
$$WP(H) = VP(H).$$ 
Pick an orthogonal normal basis $\{\xi_1, \dots , \xi_k \}$ for the Hilbert space $P(H)$, and extend it to an orthogonal  basis $B$ of $H$. 
By applying the Gram-Schmidt process to $\{W\xi_1, \dots, W\xi_k \}$, one obtains a collection 
$\{\zeta_1, \dots, \zeta_k \}$ of orthogonal unit vectors, and we extends it to an orthogonal basis $D$ of $H$. 
Now, consider $V\in \CU(H)$ satisfying $V(B) = D$ and $V\xi_i = \zeta_i$ for $i\in \{1,\dots,k \}$. 
Then one clearly has $WP(H) = VP(H)$. 
This means that $\upsilon\big([W]_{P}\big) = [V]_{P}$. 

\medskip

In the case when $\dim H < \infty$, we can actually applies the Gram-Schmidt process to the finite subset $\{W\xi:\xi\in B \}$ and obtain an element $V\in \CU(H)$ satisfying our requirement. 
In particular, this gives the following.

\medskip

\begin{eg}\label{eg:Gram-Sch}
(a)	Consider an integer $n\geq 2$. 
	Suppose that $\idp\in \CI(\BK^n)$ with $k := \dim_\BK \idp(\BK^n)$.  
	There exists $\ivt\in \GL(\BK^n)$ such that $\ivt\idp \ivt^{-1}$ is the diagonal matrix $\prj_k$ with first $k$ entries in the diagonal being $1$ and all other entries being $0$. 
	In this case, 
	$\GL(\BK^n)^{\prj_k}= \GL(\BK^k)\times \GL(\BK^{n-k})$  and $\CU(\BK^n)^{\prj_k} = \CU(\BK^k)\times \CU(\BK^{n-k})$.
	As in the above, the map $\upsilon$ from $\GL(\BK^n)$ to $\CU(\BK^n)$ given by the Gram-Schmidt process on the column vectors of matrices produces a map 
	$$\upsilon: \GL(\BK^n)/\GL(\BK^k)\times \GL(\BK^{n-k}) \to \CU(\BK^n)/\CU(\BK^k)\times \CU(\BK^{n-k})$$ 
	and this induces a locally trivial real analytic $\BK$-Banach bundle structure on $\GL(\BK^n)/\GL(\BK^k)\times \GL(\BK^{n-k})$. 
	
\smnoind
(b) Let $n\geq 2$ and $k\in \{1,\dots,n-1\}$. 
Denote by $\OGL_n$ and $\OO_n$ the sets of all $n\times n$ real invertible matrices and orthogonal matrices, respectively.
Suppose that  
$$\upsilon: \OGL_n/\OGL_k\times \OGL_{n-k} \to \OO_n/\OO_k\times \OO_{n-k}$$ 
is the map given by the Gram-Schmidt process on the column vectors of matrices. 
Then there is a bi-analytic bijection from the tangent bundle of $\OO_n/\OO_k\times \OO_{n-k}$ onto $\OGL_n/\OGL_k\times \OGL_{n-k}$. 
\end{eg}

\end{document}